\documentclass[letterpaper,twosided,11pt]{article}
\title{Computing the Cousin-Zuckerman Resolution and the Lusztig-Vogan Bijection} 
\author{Jack A. Cook}
\date{}
\usepackage{authblk}
\usepackage{amssymb}
\usepackage{amsmath}
\usepackage{stmaryrd}
\usepackage{mathrsfs}
\usepackage{hyperref} 
\usepackage{titling}
\usepackage{bbm}
\usepackage{dsfont} 
\usepackage{ bbold }
\usepackage[T1]{fontenc}
\usepackage{palatino}
\usepackage{mathpazo}
\usepackage{marvosym} 
\usepackage{tikz-cd}
\usepackage[margin=1.5in]{geometry} 
\usepackage{ragged2e}
\usepackage[utf8]{inputenc}
\usepackage[english]{babel} 
\usepackage{blindtext}

\usepackage{amsthm}
\usepackage{pgfplots}
\usepackage{tikz-3dplot}
\usepackage{pgffor}
\usepackage{comment}
\usepackage{setspace}
\usepackage{asymptote}
\usepackage{color}   
\usepackage{hyperref}
\hypersetup{
    colorlinks,
    citecolor=blue,
    filecolor=black,
    linkcolor=black,
    urlcolor=black
}
\usepackage[bottom]{footmisc}
\usepackage{graphicx}

\usetikzlibrary{shapes,calc,positioning}
\usetikzlibrary{calc,intersections}
\usetikzlibrary{bending}
	
\theoremstyle{plain} 

\newtheorem*{ntheorem}{Theorem}
\newtheorem*{nlemma}{Lemma}

\newtheorem{theorem}{Theorem}[section]

\newtheorem{lemma}[theorem]{Lemma}
\newtheorem{proposition}[theorem]{Proposition}
\newtheorem{corollary}[theorem]{Corollary}

\theoremstyle{definition}
\newtheorem{definition}{Definition}
 
\newtheorem{example}{Example}

\newtheorem*{nquestion}{Question}

\theoremstyle{remark} 
\newtheorem{remark}{Remark}

\newcommand{\C}{\mathbb{C}}

\newcommand{\HH}{\mathbb{H}}

\newcommand{\R}{\mathbb{R}}
\newcommand{\Z}{\mathbb{Z}}
\newcommand{\OO}{\mathbb{O}}

\newcommand{\KSpRes}{\widetilde{\mathcal{N}_\theta}}

\newcommand{\ip}[1]{\left \langle #1 \right \rangle}
\newcommand{\lrp}[1]{\left ( #1 \right ) }

\newcommand{\Mat}[1]{\begin{pmatrix}#1\end{pmatrix}} 
\newcommand{\script}[1]{\mathscr{#1}}
\newcommand{\close}[1]{\overline{#1}}
\newcommand{\ds}{\oplus}
\newcommand{\tensor}{\otimes}
\newcommand{\Sym}{\operatorname{Sym}}
\newcommand{\im}{\operatorname{Im }}

\newcommand{\comp}{\circ}

\newcommand{\Aut}{\operatorname{Aut}}

\newcommand{\Hom}{\operatorname{Hom}}

\newcommand{\Ind}{\operatorname{Ind}}

\newcommand{\Coind}{\operatorname{Coind}}

\newcommand{\lie}[1]{\mathfrak{#1}}

\newcommand{\ad}{\operatorname{ad}}
\newcommand{\Ad}{\operatorname{Ad}}

\newcommand{\gr}{\operatorname{gr }}

\usetikzlibrary{hobby}

\begin{document}
\maketitle

\begin{abstract} 
	The goal of this article is to give a proof of a result seemingly absent from the literature characterizing global sections of standard $\mathcal{D}$-modules on the flag variety. This characterization yields a mixture of the Langlands Classification of admissible representations with the Knapp-Zuckerman classification of tempered representations of a real reductive group. We use this result to compute the Cousin-Zuckerman resolution of the trivial representation in terms of standard $(\lie{g},K)$-modules. Further, in the case of $GL(n,\HH)$ we use this to prove the Lusztig-Vogan bijection for $n=2,3$ and compute the lowest $K$-type map for the zero and principal orbits for general $n$ as well as the image of the trivial representation for even orbits. 
\end{abstract}

\tableofcontents

\section{Introduction}

A guiding principle in the representation theory of real groups has been the Kostant-Kirillov orbit method. Roughly, this philosophy would attach to a nilpotent coadjoint orbit $\OO$ of a real Lie group a particular unitary representation $\Pi(\OO)$ satisfying certain conditions. These representations are the so called "unipotent" ones of \cite{Vogan1991}. One approach to constructing such a map $\Pi$ is outlined in \cite{Vogan1998} which we briefly recall below. One issue with this construction is that the resulting representations are not easily identified as unitary. One idea to prove this is given in \cite[Section 8]{Vogan1998}: attached to a fundamental Cartan subalgebra $\lie{h}^f$ (Definition 8.3 of \cite{Vogan1998}) we can associate a \textit{canonical real part}: a subspace on which the roots all take real values. Then for any $\lambda\in (\lie{h}^f)^*$ we can take its restriction to this canonical real part. Then we can associate to any virtual $(\lie{g},K)$-module an \textit{infinitesimal character size} which is going to be bounded by the canonical real part of the infinitesimal character. This definition then extends to orbit closures $\close{\OO}$ where we apply the preceding idea to get that the infinitesimal character of the associated graded module is bounded below by the infinitesimal character size of its associated variety. For any hermitian representation supported on an orbit closure with infinitesimal character as small as possible (i.e. equal to the infinitesimal character size of the orbit closure) must then have its negative part supported on the boundary. As we choose this representation to have minimal infinitesimal character for the given orbit $\OO,$ the negative part must be $0$ giving unitarity.

To compute the infinitesimal character size of an orbit and to thus show the anticipated representations are unitary, one needs to invoke a conjectural computable bijection between the sets

 \begin{equation}  \begin{tikzpicture}[
  every node/.style={align=center},
  arrow/.style={-Latex, thick}
  ]

\node (B) {$\resizebox{!}{.6cm}{$\left\{\begin{array}{c}
\pi\ \text{irreducible, tempered}\\
\text{with real infinitesimal character}
\end{array}\right\}$}$};

\node (A) [left=2.5cm of B] {$\resizebox{!}{.6cm}{$\left\{\begin{array}{c}
 (\OO,\mathcal{L})\\
\OO\subseteq i\lie{g}_\R^* \; \text{nilpotent } G_\R \text{-orbit}\\ \mathcal{L}\; G_\R \text{-invariant local system}\end{array}\right\}$}$};

\draw[arrow] (A.east) -- (B.west);

\end{tikzpicture}
\end{equation}
These two sets arise as bases of a certain Grothendieck group and thus abstractly correspond to one another, however giving a (canonical) bijection which respects the natural ordering on both sets is the true desire. This has been done explicitly in the case of $GL(n,\C)$ \cite{Achar2001} and abstractly for $G$ a simple complex group \cite{Bez2001}. Notice that this list does not include a real Lie group which does not admit a complex structure. One goal of this article is to give an example of a genuine real group where this bijection holds. Our methods for doing so involve exploiting results from $\mathcal{D}$-module theory in order to get a computable expression. The main theorem gives a new proof \cite[Theorem 16.5]{Vogan2007} which is better adapted to the geometry.  

Let us now be more precise about the setting. Let $G$ be a connected complex reductive algebraic group defined over $\R$ and $G_\R$ its set of real points (with corresponding map $\sigma_\R$). Let $\sigma_c:G\to G$ be a compact real form of $G.$ Set \[ \theta=\sigma_c\comp \sigma_\R:G\to G \]
 a Cartan involution of $G.$ Let $Z(G)$ denote the center of $G$ and $\Ad:G\to GL(\lie{g})$ the adjoint representation. The assumptions on $G$ guarantee that we may realize $G$ as a subgroup of $GL(n,\C)$ and that $\ker \Ad=Z(G).$ Therefore, define the \textit{adjoint group} $G_{ad}:=G/Z(G)\simeq \Ad(G).$ This can be realized as the subgroup $\operatorname{Int}(\lie{g})=\Aut(\lie{g})_0$ of inner automorphisms of $\lie{g}$.     

We now define two subgroups of interest: $K$ and $K_\theta.$  Set \begin{align*}  K=\{g\in G: \theta(g)=g\} && K_\theta=\{g\in G: \theta(g)=gz, z\in Z(G)\} \end{align*} 
Clearly $K\subseteq K_\theta$ is a proper inclusion and equality holds when $G$ is adjoint. Denote by $K_0$ the connected identity component of $K.$ By differentiating $\theta$ (and by abusing notation), we obtain an involution $\theta:\lie{g}\to \lie{g}$ which gives a decomposition \[ \lie{g}=\lie{k}\ds \lie{s} \] into the $+1$ and $-1$ eigenspaces of $\theta.$ It follows that $\lie{k}$ is the Lie algebra of $K$ (and $K_0$) and $\lie{s}$ is a $K_\theta$-invariant subspace under the adjoint representation.

 Let $\mathcal{N}$ be the cone of nilpotent elements in $\lie{g}$ and $\mathcal{N}_\theta=\mathcal{N}\cap \lie{s}.$ $K_0$ (and hence $K$ and $K_\theta$)  have finitely many orbits on $\mathcal{N}_\theta.$ Further, by \cite[Theorem 6]{KostantRallis1971}, $K_\theta$ has a unique open dense orbit $\mathbb{O}_{prin}$ on $\mathcal{N}_\theta$ whose irreducible components are given by the closures of $K_0$ orbits of elements of $\mathbb{O}_{prin}.$ This fact about components is not specific to the principal orbit. That is to say, for any $K_\theta$ (resp. $K$) orbit $\mathbb{O}$ on $\mathcal{N}_\theta,$ let $P_1,...,P_\ell$ be the finitely many maximal (with respect to the closure order) $K_0$ orbits contained in $\mathbb{O}.$ Then  \[ \close{\mathbb{O}}=\bigcup_{i=1}^\ell \close{P_i}.\]
By a result of Sekiguchi \cite{Sekiguchi1987}, the $K$-orbits on $\mathcal{N}_\theta$ are in bijection with the $G_\R$-orbits on $\mathcal{N}_{i\R}.$ Therefore, our map (1) from above can be rephrased in terms of $K$-orbits on $\mathcal{N}_\theta^*.$ Therefore, we are justified in only investigating the $K$-orbits.  

Let $X\in \mathcal{N}_\theta.$ By the Jacobson-Morozov theorem, we may complete this to an $\lie{sl}(2,\C)$-triple $\{H,X,Y\}$ which by \cite[Proposition 4]{KostantRallis1971} can be chosen with $H\in \lie{k}.$ $H$ acts on $\lie{g}, \lie{k},$ and $\lie{s}$ with integer eigenvalues and thus determines gradings $\lie{g}_i, \lie{k}_i, \lie{s}_i$ on each compatible with the Cartan decomposition \cite[Lemma 3.5]{Cook2025}. Denote by $\lie{g}^i$ (resp. $\lie{k}^i, \lie{s}^i$) the sum of all weight spaces $\lie{g}_j$ for $j\geq i.$ Taking the non-negative weights of $H$ on $\lie{g}$ we obtain a $\theta$-stable parabolic subalgebra $\lie{q}:=\lie{g}^0=\bigoplus_{i\geq 0} \lie{g}_i$ of $\lie{g}$ and hence a $\theta$-stable parabolic subgroup $Q$ of $G.$ Note that for a different choice of $X$ in $K_\theta\cdot X$ we obtain a conjugate parabolic subgroup to $Q$. Under our assumptions on $G,$ the stabilizer $G^X$ of $X$ is contained in $Q$ and thus $K^X$ is contained in $Q\cap K.$ As $Q$ is $\theta$-stable, its intersection with $K$ (and thus $K_0$) is parabolic in $K$ (and $K_0$). From this parabolic, we can construct a resolution of singularities of the orbit closure $\widetilde{\mathbb{O}}:=K\times_{Q\cap K} \lie{s}^2$ which has the structure of a homogeneous vector bundle over the projective homogeneous space $K/Q\cap K$ (see \cite[Section 3]{Cook2025} for details, note the notation is slightly different here). The proper birational map 

\begin{align*} \mu: \widetilde{\mathbb{O}}\to \close{\mathbb{O}} && \mu(k,A)=\Ad(k)A\end{align*}  is an isomorphism over the orbit.

Let $Y$ be any noetherian $G$-scheme, we write $\mathbb{K}^G(Y)$ to be the Grothendieck group of $\operatorname{Coh}^G(Y)$ the category of $G$-equivariant coherent sheaves on $Y.$ We call elements of $ \mathbb{K}^G(Y)$ \textit{virtual} coherent sheaves. If $[\mathcal{L}]\in \mathbb{K}^G(Y)$ is the class of an object in $\operatorname{Coh}^G(Y)$ we call this element \textit{genuine virtual} or simply \textit{genuine} for short. A result of Thomason \cite[Proposition 6.2]{Thomason1987} tells us that:
 
\begin{nlemma}    
	$\mathbb{K}^{K}(\widetilde{\OO})\simeq \mathbb{K}^{K\cap Q}(\lie{s}^2)\simeq \mathbb{K}^{K\cap Q}(*)$
\end{nlemma} 
The final group is identified as the representation ring of $K\cap Q.$ Restriction to an open $G$-invariant subscheme is always surjective on equivariant $\mathbb{K}$-theory. Thus we get a surjective group homomorphism \begin{equation} \mathbb{K}(K\cap Q)\to \mathbb{K}^{K}(\mu^{-1}(\OO))\simeq \mathbb{K}^{K}(\OO)\simeq \mathbb{K}(K^X)\end{equation} 
which agrees with the usual restriction map on representations. 

\begin{remark} 
	We have phrased this all in terms of $K.$ The same results hold if we instead consider $K_0.$ 
\end{remark}

We are now equipped with the notation to describe the desired map in (1). Begin with an irreducible representation $\tau\in \widehat{K^X}.$ Write $S_\tau$ the sheaf of sections of the $K$-equivariant vector bundle $K\times_{K^X} \tau$ on $\OO.$ Choose a lift $\tilde{\tau}$ of $\tau$ to $\mathbb{K}(K\cap P)$ by (2). The lift $\tilde{\tau}$ gives rise to a finite rank virtual homogeneous vector bundle $V_\tau:=K\times_{K\cap Q}\tilde{\tau}$ on $K/(K\cap Q)$ which by taking its sheaf of sections, determines a virtual coherent sheaf $\mathcal{V}_{\tilde{\tau}}$ on $K/K\cap Q.$ Let $\pi:\widetilde{\OO}\to K/K\cap P$ be the projection onto the first factor. We have that $\pi^*\mathcal{V}_{\tilde{\tau}}$ is a $K$-equivariant virtual coherent sheaf on $\widetilde{\OO}.$ As $\mu$ is proper, $R^i\mu_*$ preserves coherence. We assemble all of these derived functor sheaves into a single virtual sheaf by taking the Euler Characteristic: \[ [\textbf{R}\mu_*(\pi^*\mathcal{V}_{\tilde{\tau}})]=\sum_i (-1)^i[R^i\mu_*(\pi^*\mathcal{V}_{\tilde{\tau}})].    \]
This virtual sheaf has the following advantage: 
\begin{lemma}  $[\textbf{R}\mu_*(\pi^*\mathcal{V}_{\tilde{\tau}})]|_\OO=S_\tau.$
\end{lemma} 
To avoid such cumbersome notation, we denote $[\textbf{R}\mu_*(\pi^*\mathcal{V}_{\tilde{\tau}})]$ by $\mathcal{M}(\OO,\tilde{\tau}).$ By taking sections over the orbit closure, we obtain a virtual representation $M(\OO,\tilde{\tau})$ of $K$ which lies in $\mathcal{C}(\lie{g},K)$ (see \cite[Definition 6.8]{Vogan1998}). Thus, by \cite[Theorem 8.2]{Vogan1998} we can write \[ M(\OO,\tilde{\tau})=\sum a_{\nu,\tilde{\tau}} [\operatorname{gr} \nu] \]
for $\nu$ irreducible tempered with real infinitesimal character (tempiric). We denote the infinitesimal character of $\nu$ by $\chi(\nu)$.  As the collection on the right hand side is finite we may define \begin{align} L(M(\OO,\tilde{\tau}))=\hat{\nu}&& ||\chi(\hat{\nu})||=\max_{a_{\nu, \tilde{\tau}}\neq 0} ||\chi(\nu)||. \end{align} where $\hat{\nu}$ is a tempiric representation with the largest infinitesimal character in norm appearing in the expansion. If we choose a different lift $\tilde{\tau}',$ we will get (generically) a different largest infinitesimal character when expanded in the basis of tempered representations. So, define \begin{align}  LKT(\OO,\tau)=\nu^\star && ||\chi(\nu^\star)||= \min_{\tilde{\tau}} ||\chi(L(M(\OO,\tilde{\tau})))|| \end{align} 
The desired map is in theory given by $(\OO,\tau)\mapsto LKT(\OO,\tau).$ The output of this would be the "smallest-largest" tempiric appearing in extensions of $\tau$ to the orbit closure. The justification for this naming convention will be explained in Section 2. As the following example shows, this is sadly impossible for general real groups.    

\begin{example} 
	Let $\lie{g}=\lie{sl}(2,\C)$ and $G_\R=SL(2,\R).$ Choose $\theta$ to be \[ \theta(g)=\Mat{1 & 0 \\ 0 & -1} g \Mat{1 & 0 \\ 0 & -1}. \] This comes from the isomorphism of $SL(2,\R)\simeq SU(1,1).$ 
	Now, \begin{align*} K=\left\{ \Mat{a & 0 \\ 0 & a^{-1}}: a\in \C^\times\right\}\simeq \C^\times  &&  \lie{s}=\left\{\Mat{0 & b \\ c & 0}: b,c\in \C\right\}\end{align*}  and thus \[ \mathcal{N}_\theta=\left\{ \Mat{0 & b \\ c & 0}: bc=0\right\}.\]
	We see that $K$ is connected and has two distinct maximal orbits on $\mathcal{N}_\theta$ \begin{align*} Y^+=K\cdot \Mat{0 & 1 \\ 0 & 0} && Y^-=K\cdot \Mat{0 & 0 \\ 1 & 0} \end{align*} 
	The closures of these orbits intersect at the origin making $\mathcal{N}_\theta$ a non-normal variety. The stabilizers of the given matrix for each orbit is $\pm I.$ For $Y^+$ (resp. $Y^-$) pick $H=\Mat{1 &0 \\ 0 & -1}$ (resp. $-H$). The associated parabolic subalgebra of $\lie{g}$ attached to $H$ is \[ \lie{b}=\lie{g}_0\ds \lie{g}_2\simeq \left\{\Mat{a & b \\ 0 & -a}: a,b\in \C \right\}.\] Notice $\lie{b}\cap \lie{k}=\lie{k}$ and thus $K$ is the parabolic subgroup of $K$ attached to $H.$ 
	
	Let us first deal with the zero orbit. As this orbit is closed, each coherent $K$-equivariant sheaf supported there is necessarily just a finite dimensional representation of $K.$ Writing this in the basis of associated graded modules for tempiric representations is simple for such a small group. There are tempiric representations attached to this choice of Cartan subgroup ($K$ in this case) enumerated by non-zero integers which we denote by $D_n^{+} (resp. D_n^-)$ with lowest $K$-type $n$ (resp. $-n$) and $K$-types $n+2k$ (resp. $-n-2k$).  These are the holomorphic (for $D_n^+$) and antiholomorphic (for $D_n^-$) discrete series representations of $SL(2,\R).$ In this scheme, we see that for $n>0$ (resp $n<0$) \[[\gr D_n^+]-[\gr D_{n+2}^+]=[n] \text{ (resp. }    [\gr D_n^-]-[\gr D_{n+2}^-]=[-n])\]      
	with equality happening in $\mathbb{K}^K(\mathcal{N}_\theta).$ Following the ideas outline above, we assign \[ (\OO_0, n)\mapsto \begin{cases}  D_{n+2}^+ & n>0 \\ D_{n+2}^{-} & n<0 \end{cases} .\]
	This fills in some portion of $\Z=\widehat{K}.$ It will help to see what is happening in the following picture:  
	
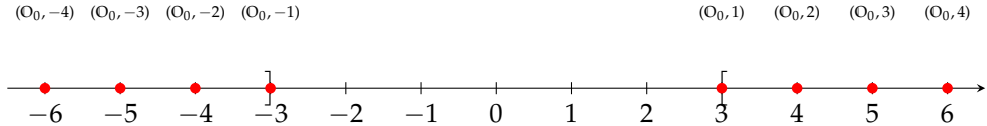
\[	\begin{tikzpicture}[>=stealth]

  \draw[->] (-6.5,0) -- (6.5,0) node[right] {};
  \foreach \x in {-6,-5,-4,-3,-2,-1,0,1,2,3,4,5,6}
    {
      \draw (\x,0.08) -- (\x,-0.08);            
      \node[below=2pt] at (\x,0) {\small $\x$}; 
    }

  \node[below=10pt, font=\bfseries] at (0,0) {};

  \node at (3,0) {\(\bigl[\)};

  \node at (-3,0) {\(\bigr]\)};
 
\foreach \x in {-6,-5,-4,-3}
{
   \pgfmathsetmacro{\y}{int(\x+2)}
  \node at (\x,1) {\tiny ($\OO_0,\y$)};
  \fill[red] (\x,0) circle (2pt);
}

\foreach \x in {3,4,5,6}
{
  \pgfmathsetmacro{\y}{int(\x-2)} 
  \node at (\x,1) {\tiny ($\OO_0,\y$)};
  \fill[red] (\x,0) circle (2pt);
}
 
\end{tikzpicture}.
	\]
	
\noindent Notice that we have omitted $n=0$ from this map. To properly deal with this trivial representation, we need to consider the final remaining tempiric representation: the spherical principal series $D_0^\pm$ of $SL(2,\R)$ which has $K$-types $0, \pm 2, \pm 4, ...$ This tempiric is attached to a final $K$-Langlands parameter for the non-split Cartan subgroup. Now, \[ [0]=[D^{\pm}_0]-[D_2^+]-[D^-_2].\] As both $D^+_2$ and $D_2^-$ have the same infinitesimal character ($\chi(\nu)=2$) there is no canonical choice for \textit{the} largest tempiric appearing.   
So, what should we send $(\OO_0,0)$ to? Before answering this, let us deal with the principal orbits.  
		
	There are two irreducible representations of $\pm I$ which we denote by $\mathbbm{1}$ (trivial)  and $\varepsilon$ (sign). Thus $\mathbb{K}(\pm I)\simeq \Z^2$ (as a $\Z$-module) and $\Z[x]/{x^2}$ as a ring.  We want to find lifts of these representations to virtual representations of $K$ whose representation ring we identify with a free $\Z$-module with basis consisting of characters $a\mapsto a^n.$ As $\operatorname{Res}_{\pm I}^K$ is only an abelian group homomorphism, we do not pay attention to the product on $\mathbb{K}(K).$ In particular, we identify:  \begin{align}  \mathbb{K}(K)\overset{\sim}{\longrightarrow}\Z[x,x^{-1}] && (a\mapsto a^n)\mapsto x^n  \end{align}  Notice that $1\in \mathbb{K}(K)$ is given by the trivial representation of $K.$ From this identification, the restriction map is given on the powers of $x$ as \[ \operatorname{Res}^K_{\pm I} (x^n)=\begin{cases}  \varepsilon & n=2k+1\\ \mathbbm{1} & n=2k	\end{cases} .\]
By some elementary linear algebra, we see that the kernel of the restriction map is the set \[ \left\{ \sum_{i=-N}^M a_ix^i: \sum_{\substack{i=-N \\ i\equiv 0\text{ mod } 2}}^M a_i=\sum_{\substack{i=-N \\ i \equiv 1\text{ mod }  2}}^M a_i=0  \right\}\]
which is clearly generated by the set of polynomials $x^i-x^j$ for $i\equiv j\mod 2.$ As an example, $2x^3-x-x^{-5}$ restricts to the zero virtual representation of $\pm I$. 

At this point, we can appeal to some facts about the orbit closures in $SL(2,\R)$ to finish the computation: the trivial representation on $Y^\pm$ should lift to the structure sheaf on $\C.$ The $K$ action ($z\in \C$ acts by $z^2)$ identifies $\C[\overline{Y^\pm}]$ with $D_2^\pm$ (the two smallest genuine discrete series). For the sign representation, we find that the extension is given by $D_1^\pm.$ We have an issue with all of this however. Notice that using the spherical principal series, we can also extend the trivial representation on either principal orbit as \[\mathcal{M}(Y^\pm,\mathbbm{1})=[D_0^\pm]-[D_2^{\mp}].\]
Note that this has the exact same largest tempiric as the trivial representation on the zero orbit! Further, the spherical principal series cannot appear as the largest tempiric in any expansion. \hfill $\blacksquare$ \end{example} 

This leaves us with a bit of a quagmire. First, the geometry of a principal $K_0$-orbit closure is generally far more complicated than adding a single point. Second, we cannot possibly get a bijection in general! This is somewhat surprising given that in \cite{Bez2001}, Bezrukavnikov gives a full proof of this conjecture for $G_\R$ a complex connected simple group. His arguments abuse the fact that for complex connected simple algebraic groups, all $K$-orbits on $\mathcal{N}_\theta$ are even dimensional and there exists a single open $K$-orbit on $\mathcal{N}_\theta.$ The first of these properties allow him to define a perverse $t$-structure on the triangulated category $D^b(\operatorname{Coh}^K(\mathcal{N}_\theta))$. The heart of this $t$-structure has simple objects consisting of the intersection cohomology sheaves $IC(\OO,\tau).$ The second, combined with properties of the Springer resolution, gave rise to an a priori different $t$-structure whose heart has simple objects parametrized by $\Lambda^+_K$ the dominant weights of $K=\Delta G_\R\cong G_\R$. It is then shown that the $t$-structures coincide and thus he obtains a bijection between the two descriptions of simple objects. It was then shown in \cite{Achar2004} that Bezrukavnikov's bijection is given by the description above in (4).    

This leaves us with a few questions: 

\begin{nquestion}
	 \color{white} a\; \color{black} 
	\begin{enumerate} 
		\item What about the geometry of the nilpotent cone for a $SL(2,\R)$ is leading to the failure of the bijection? 
		\item Do analogous $t$-structures exist for a general real reductive group?
		\item Are there any other real reductive groups for which this bijection does hold?   
		\item Can the combinatorial methods of \cite{Achar2001} and \cite{Achar2004} be mimicked for general real reductive groups to obtain a well-defined map (1)? 
	\end{enumerate} 
\end{nquestion} 

We have a partial solution to $(a)$ and $(d)$ and a positive solution for $(c).$ The main results of the paper is as follows: 

\begin{theorem} 
	Let $G_\R=GL(2,\HH)$ or $G_\R=GL(3,\HH).$ Then there is a Lusztig-Vogan bijection $(1)$ given by the map $LKT$ of $(4).$ Explicitly, for $GL(2,\HH)$ the bijection is given by \begin{align*}
 (\OO_0,(a,b))&\mapsto (a+2,b) \\ (\OO_{prin},n)&\mapsto \begin{cases}  (k,k) & n=2k \\ (k+1,k) & n=2k+1
 \end{cases} 
 \end{align*} 
 and for $GL(3,\HH)$ it is given by 

\begin{align*}
(\OO_0,(a,b,c))&\mapsto (a+4,b+2,c) \\  
(\OO_{mid}, (n_1,n_2))& \mapsto {\begin{cases} \left(n_2+2,\left \lceil \frac{n_1}{2} \right\rceil, \left \lceil \frac{n_1-1}{2} \right\rceil\right) & n_1\leq 2n_2+1 \\ \; \\ \left(\left \lceil \frac{n_1}{2} \right\rceil+1, \left \lceil \frac{n_1-1}{2} \right\rceil+1, n_2\right) & n_1>2n_2+1\end{cases}  } \\ 
 (\OO_{prin},n)&\mapsto \left( \left \lceil \frac{n}{3} \right\rceil, \left \lceil \frac{n-1}{3} \right\rceil, \left \lceil \frac{n-2}{3} \right\rceil  \right) \end{align*} 
\end{theorem} 	 
In proving the above we give a proof of a special case of the following unpublished result: 

\begin{theorem}\label{Global_Sections_Introduction}[Hecht-Mili\v{c}i\'{c}-Schimd-Wolf]
	Assume $\lambda\in \lie{h}^*$ is strongly antidominant (definition \ref{Definition_antidominant}). Let $S$ be a $K$-orbit on $\mathcal{B}$ and $\tau$ an irreducible $K$-equivariant connection compatible with $\lambda+\rho$ on $S.$ Then there exists a cuspidal parabolic subgroup $P=MAU\subseteq G$ defined over $\R,$ a (limit of) discrete series $F_{S,\tau}$ of $M_\R$ and a character $\chi_\tau$ of $A_\R$ such that \[ \Gamma(\mathcal{B}, \mathcal{I}(S,\tau))^\vee \simeq \Ind_{P_\R}^{G_\R}(F_{S,\tau}\boxtimes \chi_\tau \boxtimes \mathbbm{1})_{[K_\R]}. \]
	   
\end{theorem}
We have the following wonderful corollary: 

\begin{corollary} 
	Let $G$ be a real reductive group with single conjugacy class of Cartan subgroup. Then, under the assumptions on Theorem \ref{Global_Sections_Introduction}, $P$ is a minimal parabolic subgroup and $F_{S,\tau}$ is finite dimensional. Further, as $K_\R$ representations \[ \Gamma(\mathcal{B}, \mathcal{I}(S,\tau))|_{K_\R}\simeq \Ind_{P_\R}^{G_\R}(F_{S,\tau}\boxtimes \mathbbm{1} \boxtimes \mathbbm{1})_{[K_\R]}|_{K_\R}\]
	and the right hand side is the restriction of an irreducible tempered $(\lie{g},K)$-module with real infinitesimal character. It has a unique lowest $K_\R$-type whose highest weight is given by the $W_K$-dominant conjugate of the highest weight of $ F_{S,\tau}^\vee.$ 
\end{corollary} 

This corollary makes computing the Cousin resolution (and hence the $LKT$ map) on the zero orbit. In fact, modulo some branching problems, we obtain a formula for the pushforward of any virtual $L\cap K$-module on any orbit for $GL(n,\HH).$

\section{The Langlands Classification: Character Version}

In the example above, we abused the geometry of the cone to construct extensions of the trivial bundle to the orbit closure. In a more general setting, we do not have such luxury. Thus, we need a computational machinery to construct these extensions. To make this computation simpler, we will invoke a remarkable theorem of Vogan spread across \cite{Vogan1981} which gives the following bijection: 
	
	\begin{ntheorem}\label{Tempiric_LKT_Bijection}\cite{Vogan1981} (Theorem 11.9 in \cite{Vogan2007}) 
		Let $G_\R$ be a real reductive algebraic group with maximal compact subgroup $K_\R$ with Cartan involution $\theta.$ Then there exists a natural bijection between the following 
		\begin{description} 
			\item[(1)] Irreducible tempered representations of $G_\R$ with real infinitesimal character (Tempiric). 
			\item[(2)] Irreducible representations of $K_\R.$   
			\item[(3)] Final $K$-Pseudocharacters up to conjugation.  
		\end{description} 	 
		The bijection is given from (1) to (2) by taking the unique lowest $K_\R$-type of a given tempiric representation. We abuse notation and denote this map $LKT.$ The set of tempiric representations is denoted $\Pi(G_\R)_{temp,\R}.$ The bijection (3) to (1) is more subtle and explained below.    
	\end{ntheorem}     
	
		\begin{definition}[Definition 2.2 \cite{Vogan1984}]
			Let $G_\R$ be a real reductive group. A \textbf{limit pseudocharacter} (or $M$-character) is a quadruple \[ \gamma=(H,\Psi, \Gamma,\bar{\gamma}).\]
			satisfying the following conditions: 
			\begin{description} 
				\item[(L-1)] $H=TA$ is a $\theta$-stable Cartan subgroup of $G_\R$, $\Psi$ is a positive system for the set of imaginary roots, $\Gamma\in \widehat{H},$ and $\bar{\gamma}\in \lie{h}^*.$  
				\item[(L-2)] For $\alpha\in \Psi,$ $\ip{\alpha,\bar{\gamma}}\geq 0.$ 
				\item[(L-3)] $d\Gamma=\bar{\gamma}+\rho(\Psi)-2\rho_c(\Psi)$
			\end{description} 
			with $\rho_c$ meaning the half sum of the compact roots. 
			
			A limit pseudocharacter $\gamma$ is \textbf{final} if 
			\begin{description} 
				\item[(F-1)] If $\alpha\in \Psi$ is simple and $\ip{\alpha,\bar{\gamma}}=0$ then $\alpha$ is non-compact. 
				\item[(F-2)] If $\alpha$ is real and $\ip{\alpha,\bar{\gamma}}=0$ then $\alpha$ does not satisfy the parity condition of \cite[Definition 8.3.11]{Vogan1981}
			\end{description} 
			The set of all limit pseudocharacters will be denoted $\script{P}_{lim}(G_\R)$ and the final ones will be $\script{P}_f(G_\R).$ If we want to consider a single Cartan subgroup $H,$ then we shall substitute $G$ for $H$ in the above notation. 
		\end{definition}  
		
		Fix a limit pseudocharacter $\gamma.$ Set $M=G_\R^A$ the centralizer of $A$ in $G_\R.$ To $\gamma$ we associate a $(\lie{m},M\cap K_\R)$-module $\pi$ which is a (relative, limit of) discrete series representation. Applying parabolic induction, we obtain a standard representation \[ X(\gamma)=\Ind_{MN}^{G_\R}(\pi \tensor \mathbbm{1}).\] 
		We now have the following result of \cite{SpehVogan1980}. 
		
		\begin{theorem} 
			Let $\gamma\in \script{P}_{lim}(H).$ 
			\begin{description}
				\item[a)] $\gamma$ satisfies \textbf{(F-1)} if and only if $X(\gamma)\neq 0.$ 
				\item[b)] If $\gamma$ satisfies \textbf{(F-1)} but not \textbf{(F-2)} then $X(\gamma)$ is a direct sum of standard representations attached to more compact Cartan subgroups. 
				\item[c)] If $\gamma\in \script{P}_f(H)$ then $X(\gamma)$ has a unique irreducible submodule. 
			\end{description} 
		\end{theorem} 	
		
		Whenever $\gamma$ satisfies \textbf{(F-1)} set \[\overline{X}(\gamma)=\operatorname{soc}(X(\gamma))\]
		this is the largest completely reducible $(\lie{g},K)$-submodule of $X(\gamma).$ If $\gamma$ does not satisfy \textbf{(F-1)} then set $\overline{X}(\gamma)=0.$ 
		
		We now have the following form of the Langlands classification due to Langlands and Knapp--Zuckerman: 
		
			\begin{theorem}\label{Langlands_Classification_Pseudo_Chars}
				\begin{description} 
					\item[a)] Suppose $\gamma_i\in \script{P}_f(H_i)$ (i=1,2). Then $\overline{X}(\gamma_1)$ is equivalent to $\overline{X}(\gamma_2)$ if and only if $(H_1,\gamma_1)\sim_K (H_2,\gamma_2)$ where $\sim_K$ means $K$-conjugate. 
					\item[b)] If $X$ is any irreducible $(\lie{g},K)$-module, then there exists a $\theta$-stable Cartan $H$ and a final pseudocharacter $\gamma$ such that $X\cong \overline{X}(\gamma).$ 
				\end{description} 	
			\end{theorem} 
			
		Write $\lie{h}=\lie{t}\ds \lie{a}$ according to $\theta.$ Then, in this construction of standard modules it is clear that $X(\gamma)$ is tempered if and only if $d\Gamma|_{\lie{a}}=0.$ In this case, $X(\gamma)$ is irreducible. For a general pseudocharacter $\gamma,$ $d\Gamma|_{\lie{a}}$ controls the growth of matrix coefficients (see \cite{CasselmanMilicic1982}). 
		
		\begin{definition} 
			A $K$-\textbf{Pseudocharacter} for $G,$ is a quadruple $(H,\Psi,\Gamma,\bar{\gamma})$ satisfying:
				\begin{description} 
						\item[(K-1)] $H=TA$ is a $\theta$-stable Cartan subgroup of $G$, $\Psi$ is a positive system for the set of imaginary roots, $\Gamma\in \widehat{T},$ and $\bar{\gamma}\in \lie{t}^*.$  
				\item[(K-2)] For $\alpha\in \Psi,$ $\ip{\alpha,\bar{\gamma}}\geq 0.$ 
				\item[(K-3)] $d\Gamma=\bar{\gamma}+\rho(\Psi)-2\rho_c(\Psi)$
				\end{description} 
				A $K$-pseudocharacter $\gamma$ is \textbf{final} if 
			\begin{description} 
				\item[(KF-1)] If $\alpha\in \Psi$ is simple and $\ip{\alpha,\bar{\gamma}}=0$ then $\alpha$ is non-compact. 
				\item[(KF-2)] If $\alpha$ is real, then $\alpha$ does not satisfy the parity condition of \cite[Definition 8.3.11]{Vogan1981}. 
			\end{description} 
				The set of final $K$-pseudocharacters will be denoted $\script{P}_K.$ 
				 
		\end{definition} 	 
		 
		 Notice first that $K$-pseudocharacters are not defined off of $T$ (or are assumed to be $0$ there). Further, the finality condition requires that all real roots do not satisfy the parity condition as opposed to just those on which $\bar{\gamma}$ vanishes. This all together, defines a standard module which has the following properties: 
		 \begin{lemma}\cite[Theorem 11.9]{Vogan2007} 
		 	Let $\gamma\in \script{P}_K.$ Then  \begin{description} 
				\item[(a)]  $X(\gamma)$ is irreducible. 
				\item[(b)] $X(\gamma)$ is tempered. 
				\item[(c)] The infinitesimal character $\bar{\gamma}$ is real (in the sense of Vogan). 
				\item[(d)] $X(\gamma)$ has a unique lowest $K$-type denoted $LKT(\gamma).$ 
			\end{description} 
		 \end{lemma} 
		 
		 This gives the bijection $(1)\iff (3)$ of Theorem \ref{Tempiric_LKT_Bijection}. 
		 
\subsection{Computing the Pushforward: Tempirics and $K$-types}
		 
	We now have the necessary tools to discuss a computational approach to computing the pushforward above. With respect to the notion of "lowest" $K_\R$-type, we can construct an upper triangular matrix $m$ of multiplicities. If $\tau\in \widehat{K_\R}$ and $\pi$ is tempiric then we define \[ m(\tau, \pi)=\dim \Hom_K(\tau,\pi).\]
	As each $\pi$ has a unique lowest $K_\R$-type, this matrix is upper triangular with $1$ along the diagonal. Further, each $m(\tau,\pi)$ is a non-negative integer. Therefore, $m$ is an invertible (infinite dimensional) integer matrix. Denote by \[ M(\pi,\tau)=m^{-1}(\tau,\pi). \]
	\begin{proposition}\cite{Vogan1981}  
		The matrix $M$ satisfies the following properties: 
		\begin{description} 
			\item[(1)] $M(\pi,\tau)=0$ unless $\tau\leq LKT(\pi).$ If $\tau=LKT(\pi)$ then $M(\pi,\tau)=1.$ 
			\item[(2)] If $\tau,\tau'\in \widehat{K_\R},$ then the finite sum \[ \sum_{\pi \text{ tempiric}} m(\tau, \pi)M(\pi,\tau')=\delta_{\tau,\tau'} .\] 
			\item[(3)] If $\pi,\pi'\in \Pi(G_\R)_{temp, \R},$ then \[ \sum_{\tau\in \widehat{K_\R}} M(\pi,\tau)m(\tau,\pi')=\delta_{\pi,\pi'}.\]   
		\end{description} 
	\end{proposition} 
		Using this proposition, we immediately get a formula for irreducible representations of $K_\R$ in terms of tempiric representations. For any $\tau\in \widehat{K_\R},$ we have \[ \tau=\sum_{\pi \text{ tempiric}} M(\pi,\tau)\pi .\]
	This sum is possibly infinite by 2.1(1).  
	
	\begin{example}\label{K_as_Tempiric_Example}
	Let us investigate such a result in the case $G_\R$ is a complex reductive group treated as a real group. In this case, $K_\R$ is any compact real form of $G_\R$ and $K_\C=\Delta G_\R\cong G_\R\subset G_\R\times G_\R.$ Further, fix a maximal torus $T\subseteq K_\R.$ Then $W(K_\R,T)=W(G_\R,H)$ where $H=Z_G(T)\cong T\times T.$ Any character of $H$ extends to a character of $B_0$ (a Borel subgroup) by making the unipotent radical act trivially.  The Weyl character formula in this situation gives the desired sum of tempiric representations. Let $\tau\in \widehat{K}$ of highest weight $\xi.$ Then in $\mathbb{K}(K)$ we have \[ [\tau]=\sum_{w\in W(K,T)} \operatorname{sgn}(w)[\Ind_{B_0}^{G_\R}(\xi+(\rho_{K}-w\rho_{K}))] \]
	where $\rho_{K_\R}$ is the half sum of positive roots in $\Delta(\lie{k},\lie{t}).$ 
	   
	   Let $G_\R=GL(2,\C), K_\R=U(2),$ and $\tau=V_{(a,b)}$ the irreducible representation of $K\cong G_\R$ (equivalently of $K_\R$) with highest weight $(a\geq b)\in \lie{t}^*.$ Then $W(K_\R,T)=\Z/2\Z, \rho=(\frac{1}{2},-\frac{1}{2})$ and the Weyl character formula gives \[ V_{(a,b)}=\Ind_{B_0}^{G_\R}(a,b)-\Ind_{B_0}^{G_\R}(a+1,b-1). \]
	   We now want to write these induced representations as standard modules in the parametrization given by pseudocharacters above. Zhelobenko first deduced the classification of admissible representations of a complex group in \cite{Zhelobenko1974}. The classification scheme goes as such: $(\lie{g}_\C,K_\C)$-modules for a complex group are precisely $(\lie{g}\times \lie{g}, \Delta(G))$-modules. Such modules are thus given by pairs of characters $(\lambda_L,\lambda_R)\in \lie{h}^*\times \lie{h}^*$ such that $\lambda_L+\lambda_R\in \mathbb{X}^*(H).$ A parameter is tempered if and only if $\lambda_L=\lambda_R.$ Therefore, we see that the bijection from tempiric representations (up to conjugacy) and irreducible representations of $K$ is simply given by $I(\gamma,\gamma)\mapsto 2\gamma$ with $2\gamma\in \mathbb{X}^*(T)$ (as tempiric implies the restriction to $A$ is trivial). 
	   With this in mind, we see that the tempiric Langlands-Vogan-Zhelobenko parameters in the Weyl character formula become \begin{equation}   V_{(a,b)}=X\left(\left(\frac{a}{2},\frac b 2\right),\left(\frac a 2,\frac b 2\right)\right)-X\left(\lrp{\frac{a+1}{2},\frac{b-1}{2}},\lrp{\frac{a+1}{2},\frac{b-1}{2}}\right )  \end{equation} 
	   \hfill $\blacksquare$
	 
	\end{example} 
	
We now can use this idea to piece together the Euler characteristic of the derived pushforward. Let $\eta\in \mathbb{K}(L\cap K).$ Assume first that $[\eta]$ is the class of a genuine irreducible representation. Then according to Vogan's yoga, we write \[ [\eta]=\sum_{\pi} M(\pi, \eta)[\gr \pi] \]
with $\pi$ tempiric $(\lie{l},L\cap K)$ module. By \cite[Theorem 16.6]{Vogan2007}, we have that this sum is finite. So, define \begin{align*} M(\mathbb{O},\eta)=\sum_{i\geq 0} (-1)^i [H^0(\close{\OO},R^i (\mu)_*\pi^*(W_\eta) )] \end{align*}
As $\mathcal{N}_\theta$ is affine and $\close{\OO}$ is closed (thus affine), we have (by \cite[Proposition 8.5]{Hartshorne1977}) \[ M(\mathbb{O},\tau)\simeq \sum_{i\geq 0} (-1)^i [H^i(\widetilde{\OO},\pi^*(W_\eta) )] .\]
Further, using the projection formula \cite[\href{https://stacks.math.columbia.edu/tag/01E8}{Lemma 01E8}]{stacks-project} and that $\pi$ is affine, we see that \[ \sum_{i\geq 0} (-1)^i [H^i(\widetilde{\OO},\pi^*(W_\eta) )] = \sum_{i\geq 0} (-1)^i [H^i(K/Q_K,W_\eta \tensor \Sym(\lie{s}^2)^*)].\]
Using the De-Rham complex identity for any finite dimensional vector space $V,$ \[ [\C]=\sum_{i\geq 0} (-1)^i [\Sym V \tensor \textstyle{\bigwedge^i}V]\]
applied to $\lie{s}_1^*,$ we obtain   
  \[ \sum_{i\geq 0} (-1)^i [H^i(K/Q_K,W_\eta \tensor \Sym(\lie{s}^2)^*)]=\sum_{i\geq 0} (-1)^i [H^i(K/Q_K,W_\eta \tensor \Sym(\lie{u}\cap \lie{s})^*\tensor \textstyle{\bigwedge^i \lie{s}_1^*})] .\]

Those familiar with Blattner's formula for discrete series will notice something interesting about the right hand side above. In particular, the presence of the symmetric algebra tells us that we indeed have some form of cohomological induction appearing here. To be precise, we need to use the following result due to Zuckerman:

\begin{theorem}[Zuckerman's Blattner Formula]\cite[pg. 151]{Vogan1987}
	Let $Z$ be a fininte length $\lie{l},L\cap K$ module. Then define \[ B[\gr Z]=\sum_{i} (-1)^i[\textbf{R}^i(\mu_{\mathcal{Q}_K})_*([\gr Z]_L)\in \mathbb{K}^K(\mathcal{N}_\theta)\] (with notation as below). 
	Then \[ [\gr \mathcal{R}_\lie{q}(Z)]=B[\gr Z]. \]
\end{theorem} 

The main content of this theorem is that the pushforward along moment maps is geometrizing cohomological induction. We now need to explain the notation of the result as it will be useful later (following \cite[Section 11]{AdamsVogan2021}).   

Suppose $\lie{q}=\lie{l}\ds \lie{u}$ is a $\theta$-stable parabolic subalgebra of $\lie{g}.$ As noted before, on the group level, $Q$ being $\theta$-stable implies that $Q\cap K$ is a parabolic subgroup of $K.$ This has a Levi decomposition as \[ Q_K:=Q\cap K=(L\cap K)(U\cap K)=:L_K\cdot U_K.\] Set $\mathcal{Q}_K$ to be the variety of $\theta$-stable parabolic subalgebras conjugate to $\lie{q}$ by $K.$ 

Observe that we have a triangular decomposition $\lie{g}=\lie{q}\ds \bar{\lie{u}}$ and thus an identification (by way of a choice of bilinear form) \[\lie{q}\simeq  (\lie{g}/\lie{u})^*. \] The natural projection $\lie{q}\to \lie{l}$ corresponds to restriction of functionals \[\pi_\lie{q}:(\lie{g}/\lie{u})^*\to (\lie{q}/\lie{u})^* .\]
For this reason, an element of $\lie{q}$ is nilpotent if and only if its image in $\lie{l}$ is nilpotent. Thus, \[ \mathcal{N}_\lie{q}^*=\pi_\lie{q}^{-1}(\mathcal{N}_\lie{l}^*)\simeq \mathcal{N}_\lie{l}+\lie{u}.\] 
As $\lie{q}$ is $\theta$-stable, we may intersect with $\lie{k}$ at each step and obtain the following: \begin{align*}
	\mathcal{N}_{\lie{l},\theta}^*&=\mathcal{N}_\lie{l}^*\cap (\lie{l}/(\lie{l}\cap \lie{k}))^* \\ 
	\mathcal{N}_{\lie{q},\theta}^*&\simeq \mathcal{N}_{\lie{l},\theta}^*+(\lie{u}\cap \lie{s})^* 
\end{align*} 
We can enhance this last variety from a $Q_K$-variety to a genuine $K$-variety by forming the equivariant "affine" bundle \[ \mathcal{N}_{{\mathcal{Q}_K}}^*\simeq K\times_{Q_K} \mathcal{N}_{\lie{q},\theta}^* .\] 
Further, by abusing heavily the invariant form, we can remove the upper asterisks. We thus get a moment map \[ \mu_{\mathcal{Q}_K}: \mathcal{N}_{{\mathcal{Q}_K}}\to \mathcal{N}_\theta \] 
given by $(k,\xi)\mapsto \Ad(k)\xi.$
This map is projective but may not be surjective (there need not exist enough parabolics for nilpotent elements to vanish on). 

\begin{corollary} 
	Let $\lie{g}$ be the complexification of a real reductive Lie algebra such that there exists a single $K$-orbit on $\mathcal{N}_\theta$ whose closure is the entire cone. Let $\lie{q}$ be the associated $\theta$-stable parabolic subalgebra of $\lie{g}.$  Then $\mu_{\mathcal{Q}_K}$ is surjective. 
\end{corollary} 	   
\begin{proof} 
	Our $\mu$ defined above is the restriction of $\mu_{\mathcal{Q}_K}$ to the subbundle of $\mathcal{N}_{{\mathcal{Q}_K}}$ given by $\widetilde{\OO}.$ If there is a single orbit whose closure is the entire $K$-nilpotent cone, then $\mu$ is surjective and thus, so is $\mu_{\mathcal{Q}_K}.$ 
\end{proof} 

We need to recall the notation of cohomological induction to complete this story. Following \cite{Vogan1984}, attached to a $\theta$-stable parabolic subalgebra of $\lie{g}$ we have functors \[ \mathcal{R}_\lie{q}^i:\mathcal{M}(\lie{l},L\cap K)\to \mathcal{M}(\lie{g},K) \]
 defined as right derived functors of the Zuckerman functor composed with coinduction: \[\mathcal{R}_\lie{q}^0(V)=\Gamma_{L\cap K}^K\comp \Coind_{\lie{q}}^\lie{g}(V\tensor \textstyle{\bigwedge^{top}} \lie{u})\] with (notation as in \cite{CookCohomNotes} for Zuckerman's functor). We denote by $\mathcal{R}_\lie{q}$ the Euler-Characteristic $\sum_{i\geq 0} (-1)^i \mathcal{R}_\lie{q}^i.$ Note that we can define these functors for arbitrary Harish-Chandra pairs (again as in \cite{CookCohomNotes}) and we denote the derived functors as $R^p\operatorname{I}_{(\lie{b},T)}^{(\lie{g},K)}$\footnote{In \cite{CookCohomNotes} we do not use the convention of twisting by $\bigwedge^{top}\lie{u}.$ This introduces some differences in general, but in all cases these differences boil down to a shift in infinitesimal character.}.  We can now appropriately state a form of the pushforward.

\begin{corollary} 
	$\mu_K$ is a restriction of $\mu_{\mathcal{Q},\theta}$ to the subbundle $\KSpRes.$ Hence, \[M(\OO,\eta)|_K=\mathcal{R}_\lie{q}(W_\eta\tensor \textstyle{\bigwedge}^{top}(\lie{u}\cap \lie{s})^*)|_K.  \]
\end{corollary}

 Now the Euler-characteristic is becoming more tractable as cohomological induction plays nicely with the version of the Langlands classification given above. To be more precise, we need the following definition: 
  
 \begin{definition}\label{Good_Fair_Range}
 	Let $\lambda\in \lie{h}^*$ and $\lie{q}$ a $\theta$-stable parabolic containing $\lie{h}.$ We say that $\lambda$ is \textbf{weakly good} if \[ \operatorname{Re}(\alpha^\vee(\lambda))\geq 0\] 
	for all $\alpha\in \Delta(\lie{u},\lie{h}).$ $\lambda$ is \textbf{weakly fair} if \[  \operatorname{Re}(\alpha^\vee(\lambda|_{\lie{z}(\lie{g})}))\geq 0.\]
 \end{definition}

 These ranges of infinitesimal characters have surprising properties. First, we have the following result \cite[Proposition 4.18]{Vogan1984}:
 
 \begin{proposition}\label{cohomological_induction_pseudo_chars} 
 	Let $\gamma_{\lie{q}}=(\Psi_\lie{q},\Gamma_\lie{q},\bar{\gamma}_\lie{q})\in \script{P}^L_{lim}(H)$ be weakly-good (definition \ref{Good_Fair_Range}). Then \[ \mathcal{R}_\lie{q}^S \overline{X}^L(\gamma_\lie{q})=X^G(\gamma)\]
	where $\gamma=(\Psi_\lie{q}\cup \Delta_{im}(\lie{u},\lie{h}),\Gamma, \bar{\gamma}_\lie{q}+\rho(\lie{u}))$ with $\Gamma|_A=\Gamma_\lie{q}|_A$ and $\Gamma|_T=\Gamma_\lie{q}|_T\tensor \bigwedge^{top}(\lie{u}\cap \lie{s})|_T$, and $S=\dim \lie{u}\cap \lie{k}.$  
 \end{proposition} 
 
 In fact, we have the following result for any pseudocharacter in the weakly-fair range 
 
 \begin{proposition}\cite[Proposition 4.11]{Vogan1984}
 	Let $\gamma_\lie{q}$ and $\gamma$ be as above. Then \[ \mathcal{R}^i_\lie{q} X^L(\gamma_\lie{q})=\begin{cases} 
		0, & i\neq S\\
		X^G(\gamma),& i=S
	\end{cases} .\]
	 More generally, we have \cite[Proposition 4.13]{Vogan1984}: for all $(\lie{l},L_K)$-modules of infinitesimal character $\lambda-\rho(u),$ with $\lambda$ weakly-good, we have that $\mathcal{R}^i_\lie{q}$ vanishes unless $i=S.$  
 \end{proposition}  
 
By the second part of the previous proposition, we see that $\mathcal{R}_\lie{q}^S$ descends to a map on Grothendieck groups. This in turn has the following consequence \cite[Theorem 4.23]{SpehVogan1980}:
\begin{theorem} 
	Let $\gamma_\lie{q}\in \script{P}_{lim}^L(H)$ be weakly-fair. Write \[ X^L(\gamma_\lie{q})=\sum_{i} \overline{X}^L(\gamma_\lie{q}^i) \]
	with $\gamma_\lie{q}^i$ final parameters (attached to possibly different Cartan subgroups). Then \[ X^G(\gamma)=\sum_i \close{X}^G(\gamma^i) \]
	where $\gamma$ and $\gamma^i$ are related to $\gamma_\lie{q}$ and $\gamma_\lie{q}^i$ by Proposition \ref{cohomological_induction_pseudo_chars}.    
\end{theorem} 
 
 \begin{corollary}\label{Formula_for_Pushforward} 
 	 By writing \[\eta=\sum_{\pi \text{ tempiric} (\lie{l},L_K)} M(\pi,\eta)[\gr \pi],\] we obtain \[ M(\OO,\eta)=\sum_{j}(-1)^j\gr \mathcal{R}_\lie{q}\left(\sum_{\pi} M(\pi,\eta)\pi \tensor \textstyle{\bigwedge^j}\lie{s}_1^*\right).\]
	By replacing each $\pi$ with its corresponding Langlands parameter (in accordance with Theorem \ref{Langlands_Classification_Pseudo_Chars}), we obtain a computable formula \[ \sum_{j}(-1)^j \operatorname{gr} \mathcal{R}_\lie{q}\left(\sum_{\pi} M(\pi,\eta) \overline{X}(\gamma_\pi)\tensor \textstyle{\bigwedge^j}\lie{s}_1^*\right) .\]
 \end{corollary}


\section{The Cousin--Zuckerman Formula and Standard Modules}

We now take
up a more computational approach to the problem at hand.  Adams and
Vogan in
\cite{AdamsVogan2021} give a desideratum for what the Lusztig--Vogan
bijection should be given by.  In vague terms, it is something involving
taking lowest K-types of certain tempiric (tempered irreducible with
real infinitesimal character) representations arising from expansions in
$ \mathbb{K} $-theory.  This idea is present in
\cite{Achar2001} for $ GL(n,\C). $ Achar as well as Adams--Vogan bring
up the main issue with computing such a map.  During the step of lifting
representations from the stabilizer to the Jacobson-Morozov parabolic we
have many choices.  In some sense, Achar's algorithm computes the
``minimal'' lift by a tour de force.  We will compute something which
gives a bijection and avoids much of the brute force in other cases.
Recall the notation above, $ G_\R $ is a real reductive group, $ K_0 $ a
maximal compact in $ G_\R. $ By $ G $ and $ K $ we denote their
complexifications.


\subsection{$\mathcal{D}$-Module Methods}

The theory of $\mathcal{D}$-modules provides an efficient way 
of passing from geometry to algebra. In this section we will review 
the general theory of Beilinson-Bernstein localization and the classification
of irreducible holonomic modules. We conclude with a description of the global 
sections of standard $\mathcal{D}$-modules in terms of 
parabolic induction.


\subsubsection{Generalities on $ \mathcal{D} $-Modules}

Let $ \lie{g} $ be
a complex reductive Lie algebra $ \mathcal{B} $ denote the variety of
Borel subalgebras of $ \lie{g}. $ As demonstrated in appendix B of
\cite{Hecht1987}, we may assume $ \lie{g} $ is in fact semisimple:  this
eases the notation significantly.  Let $ \script{B} $ denote the
tautological bundle (the vector bundle on $ \mathcal{B} $ with fiber
over $ x $ equal to $ \lie{b}_x $) and $ \script{U} $ the sub-bundle
consisting of the nilpotent radicals $ \script{U}_x=[\lie{b}_x,\lie{b}_x]:=\lie
{n}_x. $ Then the quotient bundle $ \script{H}:=\script{B}/\script{U} $
is a trivial bundle over $ \mathcal{B} $ with fiber $ \lie{h} $ we call
this fiber the \textit{abstract Cartan subalgebra}%
\footnote{The introduction of this abstract Cartan subalgebra is to
avoid choosing a single basepoint in $ \mathcal{B}. $ It turns out to be
useful to have this freedom when attempting to phrase the localization
theorem and the classification.}.  We fix a root system $ \Sigma $ and a
choice of positive roots $ \Sigma^+ $ in $ \lie{h}^*. $ Set $ \rho=\frac
{1}{2}\sum_{\alpha\in \Sigma^+} \alpha. $ Let $ W $ be the Weyl group (the
group generated by reflections through the roots).  This acts naturally
on $ \mathbb{X}(\Sigma) $ the weight lattice attached to $ \Sigma. $ The
set $ \mathbb{X}(\Sigma) $ is naturally in bijection with the set of
line bundles on $ \mathcal{B} $ with the map being given by $ \lambda\mapsto
\mathcal{O}(-\lambda) $ the sheaf of sections of the $ G $-equivariant
line bundle $ G\times_B \C^*_\lambda. $ This negative twist is to be
compatible with the statements of
\cite{Cook2025}.  For proofs of statements in this section, see
\cite{MilicicLocalization,Milicic1993}.

\begin{definition}
    \label{Definition_antidominant} Let $ V $ be a finite dimensional
    vector space over $ \C $ and $ R $ a root system in $ V $.  If $
    \alpha \in R $ let $ \alpha^\vee $ denote the corresponding dual
    element in $ V^*. $ Given a choice of positive roots $ R^+, $ we say
    that $ \lambda\in V $ is \textbf{antidominant} if
    \[
        \alpha^\vee(\lambda)\leq 0, \forall \alpha\in R^+.
    \]
    We say that $ \lambda $ is \textbf{strongly antidominant} if in
    addition $
    \operatorname{Re}
    \alpha^\vee(\lambda)\leq 0 $ for all $ \alpha\in R^+. $ Further, $
    \lambda $ is regular if $ \alpha^\vee(\lambda)\neq 0 $ for all $
    \alpha\in R^+ $ (equivalently for all $ \alpha\in R $).

\end{definition}

For $ \lambda\in \lie{h}^* $ we now construct twisted sheaves of
differential operators $ \mathcal{D}_\lambda $ which are the subject of
the Localization theorem.  We follow
\cite[Section 3]{Milicic1993}.  Let $ \lie{g}^\circ:=\mathcal{O}_\mathcal
{B}\tensor_\C \lie{g}. $ We define similarly, $ \lie{b}^\circ $ and $
\lie{n}^\circ $ the subsheaves of local sections of $ \script{B} $ and $
\script{U}. $ These are sheaves of Lie algebras and we have inclusions
of subsheaves of (Lie theoretic) ideals $ \lie{n}^\circ \trianglelefteq
\lie{b}^\circ \trianglelefteq \lie{g}^\circ $.  Set $ \mathcal{U}^\circ:=\mathcal
{O}_\mathcal{B}\tensor_\C \mathcal{U}(\lie{g}). $ This is a sheaf of
associative algebras on $ \mathcal{B} $ which contains $ \lie{g}^\circ $
as a subsheaf.  Using the bracket relations, we have that $ \lie{n}^\circ\mathcal
{U}^\circ $ is a sheaf of two sided ideals in $ \mathcal{U}^\circ. $
Define
\[
    \mathcal{D}_\lie{h}=\mathcal{U}^\circ/\lie{n}^\circ \mathcal{U}^\circ.
\]
We have the following facts:

\begin{lemma}
    \cite{Milicic1993}
    \begin{enumerate}
        \item
            There is an algebra isomorphism
            \[
                \Psi:\mathcal{U}(\lie{g})\tensor_{Z(\lie{g})}\mathcal{U}
                (\lie{h})\to \Gamma(\mathcal{B},\mathcal{D}_\lie{h}).
            \]
        \item
            $ H^i(\mathcal{B},\mathcal{D}_\lie{h})=0 $ for all $ i\geq
            0. $
    \end{enumerate}
\end{lemma}

The $ Z(\lie{g}) $-algebra structure on $ \mathcal{U}(\lie{h}) $ is
given by the Harish-Chandra homomorphism $ \gamma:Z(\lie{g})\to \mathcal
{U}(\lie{h}) $ (see
\cite[Chapter 5]{Knapp2005} for details).  Further, the canonical map of
$ Z(\lie{g})\to \Gamma(\mathcal{B},\mathcal{D}_\lie{h}) $ factors
through the Harish-Chandra homomorphism%
\footnote{This tells us that $ \mathcal{D}_\lie{h} $ is a sort of
``sheafified" version of $ \mathcal{U}(\lie{h})^W. $ The precise
relationship will be made clear later.}.

Now, for any $ \lambda\in \lie{h}^*, $ we have the corresponding linear
functional $ \lambda+\rho:\lie{h}\to\C $ which extends uniquely an
algebra homomorphism $ \lambda+\rho:
\operatorname{Sym}
(\lie{h})\to \C. $ Set $ I_\lambda $ the kernel of $ \lambda+\rho. $
Then by a theorem of Harish-Chandra,
\[
    \gamma^{-1}(I_\lambda)=\gamma^{-1}(I_\mu) \iff \lambda=w\mu
\]
for some $ w\in W. $

\begin{definition}
    For $ \lambda\in \lie{h}^*, $ define
    \[
        \mathcal{D}_\lambda:=\mathcal{D}_\lie{h}/I_\lambda \mathcal{D}_\lie
        {h}.
    \]
    These are twisted sheaves of differential operators.
\end{definition}

We see immediately that $ I_{-\rho}=\lie{h}
\operatorname{Sym}
(\lie{h}) $ and thus $ \mathcal{D}_{-\rho}=\mathcal{U}^\circ/\lie{b}^\circ
\mathcal{U}^\circ $ is the sheaf of local differential operators on $
\mathcal{B}. $ If $ \lambda\in \mathbb{X}(\Sigma) $ (the weight lattice),
$ \mathcal{D}_\lambda $ is the sheaf of differential endomorphisms of $
\mathcal{O}(\lambda+\rho). $ In a similar fashion, we can consider the
global version of the above construction.  Take the ideal $ I_\lambda $
and construct
\[
    \mathcal{U}_\lambda:=\mathcal{U}(\lie{g})/\gamma^{-1}(I_\lambda)\mathcal
    {U}(\lie{g}).%
    \footnote{The notation here differs from
    \cite{Milicic1993}.  There, Mili\v{c}i\'{c} refers to this algebra
    as $\mathcal{U}_\theta$ where $\theta=W\cdot \lambda.$ We deviate
    from Mili\v{c}i\'{c} as we use the historical choice of $\theta$ to
    be the Cartan involution whereas he denotes it $\sigma.$ }
\]
By Harish-Chandra's result, we have that $ \mathcal{U}_\lambda=\mathcal{U}_\mu
$ whenever $ \lambda=w\mu $ for some $ w\in W. $ Their relationship is
given as follows:

\begin{lemma}
    \label{Global_Sections_D_lambda}
    \cite{Milicic1993}
    \begin{enumerate}
        \item
            The morphism
            \[
                \Psi:\mathcal{U}_\lambda \to \Gamma(\mathcal{B},\mathcal
                {D}_\lambda)
            \]
            is an isomorphism.
        \item
            $ H^i(\mathcal{B},\mathcal{D}_\lambda)=0 $ for all $ i\geq
            0. $
    \end{enumerate}
\end{lemma}

We now have two (left%
\footnote{We will use left modules for most of our discussion except in
one key place where it will be evident.}) module categories to consider:
$ \mathcal{M}(\mathcal{U}_\lambda) $ and $ \mathcal{M}(\mathcal{D}_\lambda).
$ It is clear that this second category is far too big to be of use to
our discussion.  Getting motivation from Serre's Theorem on
quasi-coherent sheaves for affine schemes, we look at $
\operatorname{QCoh}
(\mathcal{D}_\lambda) $ the category of $ \mathcal{D}_\lambda $-modules
which are quasi-coherent as $ \mathcal{O}_\mathcal{B} $-modules.  That
is, a $ \mathcal{D}_\lambda $-module $ \script{V} $ is quasi-coherent if
there exists an open cover $ \{U_r\} $ of $ \mathcal{B} $ such that
there is a right exact sequence
\[
    \mathcal{D}_\lambda|_{U_r}^{\ds J }\to \mathcal{D}_\lambda|_{U_r}^{\ds
    I}\to \script{V}|_{U_r}\to 0.
\]
We say that $ \script{V} $ is a \textit{coherent} $ \mathcal{D}_\lambda $
module if we have $ |J|,|I|<\infty. $ We denote the category of coherent
$ \mathcal{D}_\lambda $-modules as $
\operatorname{Coh}
(\mathcal{D}_\lambda). $ We call a quasicoherent $ \mathcal{D}_\lambda $-module
\textbf{irreducible} if it contains no non-trivial quasicoherent
submodules.

Write $ D^b(\mathcal{U}_\lambda) $ to be the bounded derived category of
$ \mathcal{U}_\lambda $-modules and denote by $ D^b_{qc}(\mathcal{D}_\lambda)
$ the bounded derived category of $
\operatorname{QCoh}
(\mathcal{D}_\lambda). $

Similar to the affine scheme case, we have a localization functor:
\begin{align*}
    \Delta_\lambda:\mathcal{M}(\mathcal{U}_\lambda)\to
    \operatorname{QCoh}
    (\mathcal{D}_\lambda)&& \Delta_\lambda(V)=V\tensor_{\mathcal{U}_\lambda}\mathcal
    {D}_\lambda.
\end{align*}
Additionally, by Lemma \ref{Global_Sections_D_lambda} (a), we have a
global sections functor
\begin{align*}
    \Gamma:
    \operatorname{QCoh}
    (\mathcal{D}_\lambda)\to \mathcal{M}(\mathcal{U}_\lambda)&& \Gamma(\script
    {V})=\Gamma(\mathcal{B},\script{V}).
\end{align*}
These are easily shown to be adjoint functors with $ \Gamma $ the right
adjoint.  In fact, for certain $ \lambda $ we can say more:

\begin{theorem}
    \cite{BeilinsonBernstein1981} Let $ \lambda\in\lie{h}^* $ be
    regular.  Then
    \[
        \textbf{L}\Delta_\lambda:  D^b(\mathcal{U}_\lambda)
        \rightleftarrows D^b_{qc}(\mathcal{D}_\lambda):\textbf{R}\Gamma
    \]
    is an equivalence of triangulated categories.  If in addition $
    \lambda $ is antidominant, then $ \Delta_\lambda $ is exact and is
    an equivalence of categories between $ \mathcal{M}(\mathcal{U}_\lambda)
    $ and $ \mathcal{M}(\mathcal{D}_\lambda). $ Further, this
    equivalence restricts to an equivalence between $ \mathcal{M}_{fg}(\mathcal
    {U}_\lambda) $ and $
    \operatorname{Coh}
    (\mathcal{D}_\lambda). $
\end{theorem}

For the purposes of representation theory, we need to also consider $ K $-equivariant
analogs of the above theorem.  To properly state what we will need
later, let us set up some more general notation.  For any $ \script{V}\in
\operatorname{Coh}
(\mathcal{D}_\lambda) $ there exists a (non-unique) \textit{good
filtration} (see
\cite{MilicicLocalization}) which we will denote by $ F_\bullet \script{V}.
$ Embedded in the definition of these filtrations, there is a
compatibility condition with the standard degree filtration on $
\mathcal{D}_\lambda $ making $ \gr_F \script{V} $ a coherent $ \gr
\mathcal{D}_\lambda $-module.  The graded ring $ \gr \mathcal{D}_\lambda
$ does not depend on the twist given by $ \lambda. $ Denote by $ \pi_\mathcal
{B} $ the canonical projection $ T^*\mathcal{B}\to \mathcal{B}. $ We
have an isomorphism
\[
    \gr \mathcal{D}_\lambda\simeq (\pi_\mathcal{B})_*\mathcal{O}_{T^*\mathcal
    {B}}.
\]
This gives the following result:

\begin{lemma}
    For $ \script{V}\in
    \operatorname{Coh}
    (\mathcal{D}_\lambda), $ there exists a unique $ \widetilde{\script{V}}\in
    \operatorname{Coh}
    (\mathcal{O}_{T^*\mathcal{B}}) $ such that $ (\pi_\mathcal{B})_*(\widetilde
    {\script{V}})= \gr_F \script{V}. $
\end{lemma}

\begin{definition}
    For $ \script{V}\in
    \operatorname{Coh}
    (\mathcal{D}_\lambda), $ the \textbf{characteristic variety} of $
    \script{V} $ is
    \[
        \operatorname{Ch}
        (\script{V}):=
        \operatorname{supp}
        (\widetilde{\script{V}}).
    \]
    A theorem of Gabber
    \cite[Theorem 1]{Gabber1981} proves that $
    \operatorname{Ch}
    (\script{V}) $ is a coisotropic subvariety of $ T^*\mathcal{B} $ for
    the canonical symplectic structure.  This implies that
    \[
        \dim
        \operatorname{Ch}
        (\script{V})\geq \dim \mathcal{B}.
    \]
    We say $ \script{V} $ is \textbf{holonomic} if $ \dim
    \operatorname{Ch}
    (\script{V})=\dim \mathcal{B}. $ This can all be carried out on
    arbitrary smooth algebraic varieties and we will use this fact
    freely.
\end{definition}

The usual $ \mathcal{O}_\mathcal{B} $-theoretic functors of pushforward
and pullback are inadequate for studying $ \mathcal{D}_\lambda $-modules.
The analogs have the deficit of being ``sided" (in the sense that they
only respect left or right $ \mathcal{D}_\lambda $-modules).

\begin{definition}
    Put $ \mathcal{M}^L(R) $ denote the category of left modules over a
    ring (or sheaf of rings).  For a smooth algebraic variety $ X, $
    write $ \mathcal{D}_X $ as the sheaf of differential endomorphisms
    of $ \mathcal{O}_X. $ If $ f:X\to Y $ is a morphism of smooth
    algebraic varieties, define the \textbf{transfer module}
    \[
        \mathcal{D}_{X\to Y}=f^*(\mathcal{D}_Y)=\mathcal{O}_X\tensor_{f^
        {-1}\mathcal{O}_Y} f^{-1}\mathcal{D}_Y.
    \]
    This is a $ (\mathcal{D}_X,f^{-1}\mathcal{D}_Y) $-bimodule.

    For $ \script{V}\in \mathcal{M}^L(\mathcal{D}_Y), $ define the ($
    \mathcal{D} $-module) \textbf{inverse image} of $ \script{V} $ to be
    \[
        f^+(\script{V})=\mathcal{D}_{X\to Y} \tensor_{f^{-1}\mathcal{D}_Y}
        f^{-1}\script{V}.
    \]
    As an $ \mathcal{O}_X $-module, this is the same as $ f^*\script{V}.
    $ Sadly, the direct image is not as elegant.

    For $ \script{V}^\bullet \in D(\mathcal{D}_X) $ (unbounded derived
    category of right modules), the \textbf{direct image} of $ \script{V}^\bullet
    $ is
    \[
        f_+(\script{V}^\bullet)=\textbf{R}f_*(\script{V}^\bullet \tensor^
        {\textbf{L}}_{\mathcal{D}_X} \mathcal{D}_{X\to Y}).
    \]
\end{definition}

Although these definitions are a bit obscure, they have the following
excellent properties:

\begin{theorem}
    \cite[V.3.1]{MilicicDModules} Denote by $ L^pf^+ $ and $ R^pf_+ $
    the cohomology of the inverse and direct image functors.  Then $ L^pf^+
    $ and $ R^pf_+ $ send holonomic modules to holnomic modules.
\end{theorem}

Now let $ X $ be a smooth algebraic variety and $ Y\subseteq X $ a
smooth closed subvariety.  We would like to know if there is a relation
between $ \mathcal{D}_X $-modules supported in $ Y $ and $ \mathcal{D}_Y
$-modules on their own.  Let $ i:Y\to X $ be the closed embedding.  As $
i $ is proper, $ i_*=i_! $ and $ i_* $ is exact.  Further, $ \mathcal{D}_
{Y\to X} $ is a locally free $ \mathcal{D}_Y $-module.  Therefore, $ i_+
$ is exact and takes right $ \mathcal{D}_Y $-modules to right $ \mathcal
{D}_X $-modules without needing to pass to derived categories.  In this
situation, Kashiwara has the following wonderful result:

\begin{theorem}[Kashiwara]
    The functor $ i_+ $ induces an equivalence of categories
    \[
        \mathcal{M}(\mathcal{D}_Y)\to \mathcal{M}_Y(\mathcal{D}_X)
    \]
    between right $ \mathcal{D}_Y $-modules and $ \mathcal{D}_X $-modules
    supported on $ Y. $
\end{theorem}

There is a particular class of $ \mathcal{D}_X $-modules that this
theorem will apply to.  Note that a holonomic $ \mathcal{D}_X $-module
is one whose charactersitic variety is a Lagrangian subvariety of $ T^*X.
$.  A large class of Lagrangian subvarieties of cotangent bundles are
conormal bundles.  The most special conormal bundle is the zero-section
(conormal to all of $ X $!).

\begin{lemma}
    \cite[V.2]{MilicicDModules} The following statements are equivalent
    for a $ \mathcal{D}_X $-module $ \script{V}: $

    \begin{enumerate}
        \item
            $
            \operatorname{Ch}
            (\script{V}) $ is the zero section.
        \item
            $ \script{V} $ is a coherent $ \mathcal{O}_X $-module.
        \item
            $ \script{V} $ is a finite locally free $ \mathcal{O}_X $-module.
    \end{enumerate}
    If $ \script{V} $ satisfies any of these properties, we call $
    \script{V} $ a \textbf{connection}.
\end{lemma}

Attached to any connected subvariety $ Y $ of $ \mathcal{B} $ and
connection $ \tau $ on $ Y $ we assign a \textbf{standard} $ \mathcal{D}_\lambda
$-module
\[
    \mathcal{I}(Y,\tau):=i_+(\tau).
\]

\begin{lemma}
    $ \mathcal{I}(Y,\tau) $ has the following properties:
    \begin{enumerate}
        \item
            $
            \operatorname{supp}
            (\mathcal{I}(Y,\tau))=\close{Y}. $
        \item
            $ \mathcal{I}(Y,\tau) $ has a unique irreducible submodule $
            \mathcal{L}(Y,\tau). $
    \end{enumerate}
\end{lemma}

We are now fit to classify the irreducible holonomic $ \mathcal{D}_\mathcal
{B} $-modules.  Suppose that $ \script{V} $ is an irreducible holonomic
module.  Let $ S=
\operatorname{supp}
(\script{V}) $ denote its support.  This is a closed subvariety of $
\mathcal{B} $ and is irreducible as $ \script{V} $ is irreducible.  Let $
i:S:\to \mathcal{B} $ denote the inclusion of this closed subvariety.
There exists an open affine subvariety $ U $ of $ \mathcal{B} $ such
that $ V=
\operatorname{supp}
(\script{V})\cap U $ is a closed smooth subvariety of $ U. $ Clearly, $
\script{V}|_U $ is irreducible and has support $ V. $ By Kashiwara's
theorem, there exists some irreducible holonomic $ \script{W} $ such
that $ (i|_V)_+(\script{W})=\script{V}|_U. $ Further, $
\operatorname{supp}
(\script{W})=V. $ Therefore, by the lemma, there exists a dense open
subset $ A\subseteq V $ such that $ \script{W}|_A $ is a connection.
Hence, by shrinking $ U $ if necessary, we can assume $ \script{W} $ is
a connection.

We now have two irreducible submodules of $ (i|_V)_+(\script{W})|_U: $ $
\script{V}|_U $ and $ \mathcal{L}(V,\script{W})|_U. $ By uniqueness,
these must be isomorphic.  It then follows that $ \script{V}\simeq
\mathcal{L}(V,\tau). $ Thus, we have the following:

\begin{theorem}
    Let $ \script{V} $ be an irreducible holonomic $ \mathcal{D}_\mathcal
    {B} $-module.  Then there exists an irreducible smooth open subset $
    U $ of the support of $ \script{V} $ and an irreducible connection $
    \tau $ on $ U $ such that $ \script{V}\simeq \mathcal{L}(U,\tau). $
\end{theorem}

This theorem is nearly what we need for representation theory
considerations.  However, the $ K $-action is missing.  Recall that in
our setting, we have that $ K $ act by finitely many orbits on $
\mathcal{B}. $ Further, for simplicity of the notation, we assume $ K $
is connected (again this can be removed, see
\cite{Hecht1987} appendix B).

Returning to the equivalence of categories above, consider the two new
categories $ \mathcal{M}_{fg}(\mathcal{U}_\lambda,K) $ and $
\operatorname{Coh}
(\mathcal{D}_\lambda,K) $ consisting of finitely generated $ (\lie{g},K)
$-modules with infinitesimal character $ \lambda $ and $ K $-equivariant
coherent $ \mathcal{D}_\lambda $-modules respectively.

\begin{definition}
    The objects of $
    \operatorname{Coh}
    (\mathcal{D}_\lambda,K) $ are called \textbf{Harish-Chandra sheaves}.
\end{definition}

Clearly, the image of the localization functor $ \Delta_\lambda $ on $
\mathcal{M}_{fg}(\mathcal{U}_\lambda,K) $ lands in $
\operatorname{Coh}
(\mathcal{D}_\lambda,K). $ Further, the cohomology of these sheaves are
finitely generated $ \mathcal{U}_\lambda $-modules.

\begin{lemma}
    \cite[pg.  88]{MilicicLocalization} Every Harish-Chandra sheaf
    admits a $ K $-equivariant good filtration.  Further, its
    characteristic variety is contained in the union of all conormal
    bundles to $ K $-orbits in $ \mathcal{B}. $ In particular,
    Harish-Chandra sheaves are holonomic.
\end{lemma}

The final piece we need to fully classify the irreducible Harish-Chandra
sheaves is a characterization of the irreducible connections which are $
\mathcal{D}_\lambda $-modules.  Suppose $ \tau $ is an irreducible
connection on a $ K $-orbit $ O. $ Let $ x\in O $ be any point.  As $
\tau $ is irreducible, the geometric fiber $ T_x\tau $ is an irreducible
finite dimensional representation of the stabilizer $ S_x=\{k\in K:
kx=x\}. $ In fact, $ \tau $ is completely determined by this
representation $ \omega $ of $ S_x. $ Let $ \lie{c} $ be a $ \theta $-stable
Cartan subalgebra contained in $ \lie{b}_x. $ Then the Lie algebra of $
S_x $ is a semi-direct product of the toroidal part $ \lie{t} $ of $
\lie{c} $ with the nilpotent radical $ \lie{u}_x=\lie{k}\cap \lie{n}_x. $
Let $ T $ and $ U $ be the Levi factor and unipotent radical of $ S_x $
corresponding to $ \lie{t} $ and $ \lie{u}_x $ respectively.  Then as $
\omega $ is irreducible, $ U $ acts trivially and thus $ \omega $ can be
viewed as an irreducible representation of the reductive group $ T. $
The differential of this action is necessarily then a direct sum of
copies of the specialization of $ \lambda+\rho $ restricted to $ \lie{t}
$ (as $ \lie{t} $ is abelian and contained in $ \lie{c} $).  In this
case, we say that $ \tau $ is \textit{compatible with} $ \lambda+\rho. $

\begin{lemma}
    \label{Classification_of_Connections} Let $ O $ be a $ K $-orbit in $
    \mathcal{B} $ and $ \lambda $ a weight of a finite dimensional
    representation of $ K. $ The irreducible connections compatible with
    $ \lambda+\rho $ on $ O $ stand in one-to-one correspondence with
    irreducible representations of $ \pi_0(S_x) $ (the component group
    of the stabilizer of any point in $ x $).
\end{lemma}
\begin{proof}
    As we assume $ G_\R $ is the real points of a connected reductive
    complex algebraic group, $ G_\R $ is linear.  In this case, $ T $ is
    contained in a complex torus in $ G $ and is hence abelian.
    Therefore, $ \omega $ is one-dimensional and is the restriction of a
    single copy of $ \lambda+\rho. $

    If $ S_x $ is connected, this determines the representation
    entirely.  As $ U $ is always connected, the connectedness of $ S_x $
    is controlled by $ T. $ Suppose $ T $ is disconnected.  Then as an
    algebraic group $ T $ has finitely many connected components.
    Therefore $ T=T_0F $ for a finite abelian group $ F $ and $ \pi_0(T)=T/T_0\simeq
    F. $ We can extend every irreducible representation $ \kappa $ of $
    F $ to a representation of $ T $ by making it act as $ \lambda+\rho $
    on $ T_0. $ Thus, to any such $ \kappa $ we obtain a connection $
    \tau_\kappa $ on $ O $ which as a $ K $-equivariant vector bundle is
    given by
    \[
        \tau_\kappa(-)=\Gamma(-, K\times_{S_x} \kappa).
    \]
    This finishes the proof.

\end{proof}

We now have the capacity to fully classify Harish-Chandra sheaves.

\begin{theorem}
    \cite{BeilinsonBernstein1993} Let $ \script{V}\in
    \operatorname{Coh}
    (\mathcal{D}_\lambda,K) $ be irreducible.  Then there exists a $ K $-orbit
    $ S $ and an irreducible connection compatible with $ \lambda+\rho, $
    such that $ \script{V}=\mathcal{L}(S,\tau). $ Further, if $ \mathcal
    {L}(S,\tau)\simeq \mathcal{L}(S',\tau') $ then $ S=S' $ and $ \tau\simeq\tau'.
    $
\end{theorem}


\subsubsection{The Cousin Resolution}

We now set up the main computational tool for computing the extensions:
the Cousin Resolution.  Earlier treatments of this result due to Kempf (
\cite{Kempf1978}) are given in terms of the Bruhat decomposition of
homogeneous spaces ($ B $-orbits on $ \mathcal{B} $).  The version (and
its applications) we give here is in terms of a stratification on a
smooth variety.  In our application, we use the $ K $-orbits on $
\mathcal{B} $ to give a more explicit description of the resolution in
terms of standard $ \mathcal{D}_\lambda $-modules.  For proofs, see
appendix A of
\cite{MilicicRomanov2025}

\begin{definition}
    Let $ X $ be a smooth algebraic variety.  A \textbf{stratification
    of length n} on $ X $ consists of a decreasing filtration
    \[
        X=F_0\supseteq F_1\supseteq ...  \supseteq F_n\supseteq F_{n+1}=\varnothing.
    \]
    by closed algebraic subvarieties satisfying the condition
    \begin{description}
        \item[(S)]
            $ F_p-F_{p+1} $ is a nonempty smooth affinely imbedded
            subvariety of $ X $ of dimension $ \dim X-p $ for all $ 0\leq
            p\leq n. $
    \end{description}
    The subvarieties $ S_p=F_p-F_{p+1} $ are called the $ (\dim X-p) $-dimensional
    strata.
\end{definition}

A stratification $ F_p $ on $ X $ induces stratifications of the smooth
subset $ X_q=X-F_q $ of length $ q. $ Now let $ Y\subseteq X $ be a
closed smooth subvariety and $ U=X-Y $ its complement denote by $ i:Y\to
X $ and $ j:U\to X $ the inclusions.  As $ j $ is an open immersion, $ j_+=j_*.
$ As $ Y $ is closed, $ i_*=i_! $ is exact and thus $ i_+ $ is an exact
functor on the level of the module categories.  Denote by $ i^! $ the
right adjoint of $ i_+ $ (see
\cite[IV.8]{MilicicDModules} for a definition).  We then see (by some
adjointness properties) that
\[
    \textbf{R}i^!\simeq \textbf{R}i^*[\dim Y-\dim X].
\]

\begin{lemma}
    Let $ m=\dim X-\dim Y. $ Then
    \begin{enumerate}
        \item
            For $ m=1, $ we have an exact sequence
            \[
                0\to \mathcal{O}_X\to j_*(\mathcal{O}_U)\to i_+(\mathcal
                {O}_Y)\to0.
            \]
        \item
            For $ m>1, $ we have that
            \[
                R^pj_*(\mathcal{O}_U)=
                \begin{cases}
                    \mathcal{O}_X & p=0\\
                    i_+(\mathcal{O}_Y) & p=m-1 \\
                    0 & else
                \end{cases}.
                .
            \]
    \end{enumerate}
\end{lemma}

Using this lemma, we now have the main tool in the computation:  the
Cousin resolution.

\begin{theorem}
    Let $ X $ be a smooth variety with stratification $ \{F_i\}_{i\in I}.
    $ Denote by $ S_p=F_p-F_{p+1} $ and $ i_p:S_p\to X $ the inclusion
    of the strata.  Then there exists a canonical exact sequence of
    quasi-coherent $ \mathcal{D}_X $-modules:
    \[
        0\to \mathcal{O}_X\to (i_0)_+(\mathcal{O}_{S_0})\to ...  \to (i_n)_+
        (\mathcal{O}_{S_n})\to 0.
    \]
\end{theorem}

Applying this to $ K $-orbits needs some discussion.  Firstly, $ K $
acts on $ \mathcal{B} $ with finitely many orbits
\cite{Wolf1974}.  These orbits are affinely imbedded by
\cite[Chapter 4, Proposition 1.1]{MilicicLocalization}.  Further, there
exist $ K $-orbits of each dimension $ p $ from $ \dim X $ to the
minimal dimension over all orbits.  Thus, we can build a stratification
by collecting all orbits of the same dimension into one strata.  For $ Y
$ a smooth connected subvariety of $ X $, $ \mathcal{O}_Y $ is an
irreducible $ \mathcal{D}_Y $-module.  Further, it is coherent as an $
\mathcal{O}_Y $-module and thus is a connection on $ Y. $ Set $ \mathcal
{I}(Y)=i_+(\mathcal{O}_Y) $ the standard $ \mathcal{D}_X $-module
attached to the structure sheaf.  This gives the following corollary for
the flag variety:
\begin{corollary}
    \label{Cousin_Resolution_Corollary} Let $ Y\subseteq X $ be a a $ K $-orbit
    and suppose $ Z=\overline{Y} $ is smooth.  Then we have a canonical
    exact sequence:
    \[
        0\to \mathcal{L}(Y)\to \mathcal{I}(Y) \to \bigoplus_{%
        \operatorname{codim}
        O=1} \mathcal{I}(O) \to ...\to \bigoplus_{%
        \operatorname{codim}
        O=\dim Z-p} \mathcal{I}(O)\to 0
    \]
    where $ p $ is the minimal dimension of a $ K $-orbit in $ Z. $
\end{corollary}

We can instead phrase all of this in terms of $ \mathcal{D}_\lambda $-modules
by replacing the structure sheaves with connections compatible with $
\lambda+\rho. $ The following corollary results from taking global
sections.

\begin{corollary}[Zuckerman Character formula]
    Let $ \mathbb{K}_0(\mathcal{U}_\lambda) $ denote the Grothendieck
    group of the category of $ \mathcal{U}_\lambda $-modules.  Set $ p=\displaystyle\min_
    {K-orbits} \dim O. $ Then in $ \mathbb{K}_0(\mathcal{U}_0) $ we have
    the following identity:
    \[
        [\C]=\sum_{i=0}^{\dim \mathcal{B}-p}(-1)^i\sum_{%
        \operatorname{codim}
        O=i}[\Gamma(\mathcal{B},\mathcal{I}(O))].
    \]
\end{corollary}

We have an even more specific version when $ G_\R $ has a single
conjugacy class of Cartan subgroups.  Recall the definition of $ A_\lie{b}
(0) $ from
\cite{VoganZuckerman1984}.

\begin{corollary}
    Suppose $ G_\R $ has a single conjugacy class of Cartan subgroups.
    Then the character formula takes the form:
    \[
        [\C]=\sum_{i=0}^{\dim \mathcal{B}-p-1}(-1)^i\sum_{
        \operatorname{codim}
        O=i}[\Gamma(\mathcal{B},\mathcal{I}(O))]+[A_\lie{b}(0)].
    \]
\end{corollary}
\begin{proof}
    The only additional work to be done here is to show that
    \begin{enumerate}
        \item
            There is a single closed $ K $-orbit $ O_{cl} $ on $
            \mathcal{B}. $
        \item
            $ \Gamma(\mathcal{B},\mathcal{I}(O_{cl}))=A_\lie{b}(0). $
    \end{enumerate}

    For $ (a), $ we recall some general facts about $ K $-orbits.  To
    each $ K $-orbit $ S $ we can associate a $ K $-conjugacy class of $
    \theta $-stable Cartan subalgebras determined by
    \[
        x\mapsto \Ad K\cdot \lie{c}
    \]
    for $ \lie{c} $ a $ \theta $-stable Cartan subalgebra in $ \lie{b}_x
    $ (see
    \cite[Lemmas 5.3-4]{Milicic1993}.  This map is surjective and the
    fiber of this map over any fixed $ \theta $-stable Cartan subalgebra
    is parametrized by $ W_K $-conjugacy classes of choices of positive
    roots in $ \Delta(\lie{g},\lie{c}). $

    In the context of the corollary, we assume $ G_\R $ has a single
    conjugacy class of Cartan subalgebras.  Therefore, there is a single
    $ K $-conjugacy class of $ \theta $-stable Cartan subalgebras by
    Matsuki.  By the above, we thus have that every $ K $-orbit
    is attached to the same conjugacy class and thus a bijection
    \[
        \{ K-\text{orbits}\} \overset{\sim}{\longleftrightarrow} \left\{\parbox
        {4cm}{%
        \centering
        $W_K$ -conjugacy classes of positive roots in $\Delta(\lie{g},\lie
        {c})$ }\right\}.
    \]
    Let $ w_0^K $ and $ e^K $ be the unique longest element and identity
    element of $ W_K $ respectively.  It is known that the closure order
    on $ K $-orbits then coincides with the reverse Bruhat order on $ W_K
    $
    \cite{RichardsonSpringer1990}.  Therefore, there are unique $ K $-orbits
    corresponding to $ w_0^K $ and $ e^K $ respectively.  The orbit
    attached to $ w_0^K $ is the unique open orbit on $ \mathcal{B} $
    and the orbit attached to $ e^K $ is the unique closed orbit.  This
    prove (a).

    The proof of $ (b) $ is a rather straightforward application of the
    duality theorem of
    \cite{Hecht1987}.  Fix a $ \theta $-stable Cartan subalgebra $ \lie{c}
    $ of $ \lie{g}. $ Write $ C $ for the corresponding connected
    subgroup of $ G. $ Set $ T=C^\theta. $ Let $ S $ be a $ K $-orbit in
    $ \mathcal{B} $ and $ x\in S $ such that $ \lie{c}\subseteq \lie{b}_x.
    $ Set $ s=\dim( \lie{k}\cap \lie{n}_x). $ Denote by $ \tau $ a
    connection on $ S $ compatible with $ \lambda+\rho. $ In our
    notation, that theorem says that
    \[
        H^p(\mathcal{B},\mathcal{I}(S,\tau))^\vee\simeq R^{s-p}
        \operatorname{I}
        _{(\lie{b}_x,T)}^{(\lie{g},K)} (\tau_x^\vee\tensor T_x\omega_{\mathcal
        {B}}).
    \]
    For application to the Cousin resolution, we simply consider $ \tau=\mathcal
    {O}_{O_{cl}} $ the structure sheaf of the closed orbit.  This is a $
    \mathcal{D}_\mathcal{B}=\mathcal{D}_{-\rho} $-module and is
    compatible with $ -\rho+\rho=0 $ as the the differential of the
    representation of the stabilizer $ S_x $ on $ (\mathcal{O}_{O_{cl}})_x
    $ is the trivial representation (see Lemma \ref{Classification_of_Connections}).
    As $ O_{cl} $ is closed, standard $ \mathcal{D}_\lambda $-modules
    attached to it are irreducible.%
    \footnote{The quotient $ \mathcal{I}(S,\tau)/\mathcal{L}(S,\tau) $
    is supported on the boundary of the orbit $ S $.} Lastly, any $ x\in
    O_{cl} $ correspond to $ \theta $-stable Borel subalgebras.  Hence,
    we have that $ \mathcal{I}(S,\mathcal{O}_{O_{cl}})=\mathcal{L}(S,\mathcal
    {O}_{O_{cl}}) $ and
    \[
        H^p(\mathcal{B},\mathcal{I}(S,\mathcal{O}_{O_{cl}}))\simeq R^{s-p}
        \operatorname{I}
        _{(\lie{b}_x,T)}^{(\lie{g},K)} (\C\tensor T_x\omega_{\mathcal{B}})
        = \mathcal{R}_{\lie{b}_x}^{s-p}(\C_0).
    \]
    Taking $ p=0, $ we see the right hand side is exactly the definition
    of $ A_{\lie{b}_x}(0). $

\end{proof}

In proving $ (a) $, we showed the following holds:
\begin{corollary}
    Suppose $ G_\R $ has a single conjugacy class of Cartan subgroups.
    Then all $ \theta $-stable Borel subalgebras of $ \lie{g} $ are $ K $-conjugate.
\end{corollary}
\begin{proof}
    The variety of $ \theta $-stable Borel subalgebras is a union of the
    closed $ K $-orbits
    \cite[Lemma 6.16]{Hecht2022}.  There is a unique closed $ K $-orbit
    in this situation.
\end{proof}


\subsection{Global Sections of Standard $ \mathcal{D}_\lambda $-Modules}

The standard modules constructed above are well suited for proving
theorems using the combinatorial description of pseudocharacters.  We
would like to compare this to the geometric construction which can be
best adapted to the situation of the nilpotent cone.  In particular, we
want to make the Cousin--Zuckerman resolution computable.  We will (mostly)
follow the notation of
\cite{Hecht1987} and
\cite{Hecht2022}.  I thank Dragan Mili\v{c}i\'{c} for explaining the
theory of $ \mathcal{D} $-modules to me and alerting me to the following
unpublished result:

\begin{theorem}(Hecht--Mili\v{c}i\'{c}--Schmid--Wolf) 
    \label{Global_Sections_from_D} Let $ S $ be a $ K $-orbit on $
    \mathcal{B} $ and $ \tau $ an irreducible $ K $-equivariant
    connection compatible with $ \lambda+\rho $ on $ S. $ Then there
    exists a cuspidal parabolic subgroup $ P=MAU\subseteq G $ defined
    over $ \R, $ a (limit of) discrete series $ F_{S,\tau} $ of $ M_\R $
    and a character $ \chi_\tau $ of $ A_\R $ such that
    \[
        \Gamma(\mathcal{B}, \mathcal{I}(S,\tau))^\vee \simeq \Ind_{P_\R}^
        {G_\R}(F_{S,\tau}\boxtimes \chi_\tau \boxtimes \mathbbm{1})_{[K_\R]}.
    \]

\end{theorem}

This has a specialization for the case of a single conjugacy class of
Cartan subgroups.

\begin{theorem}\label{Single_Conjugacy_Class_Sections} 
    Let $ G_\R $ be a real reductive group with single conjugacy class
    of Cartan subgroups.  Let $ \lie{h} $ be the abstract Cartan
    subalgebra.  Let $ S $ be a $ K $-orbit on $ \mathcal{B}. $ Let $ P_\R=M_\R
    A_\R N_\R $ be the Langlands decomposition of a minimal parabolic
    subgroup of $ G_\R. $ Assume $ \lambda\in \lie{h}^* $ is strongly
    antidominant, then
    \[
        \Gamma(X,\mathcal{I}(S,\tau))^\vee\cong \Ind_{P_\R}^{G_\R}(\mu\boxtimes
        \eta\boxtimes \mathbbm{1})_{[K_\R]}
    \]
    with $ \mu $ a finite dimensional irreducible representation of $ M_\R
    $ and $ \eta $ an irreducible representation of $ A_\R. $
\end{theorem}
Before giving a proof of this special version, we need to recall some
 fundamental constructions from homological algebra and the structure 
 theory of reductive groups. We start with structure theory: 
 
 Let $ S $
be a $ K $-orbit on $ \mathcal{B}, $ $ x\in S $ and $ \lie{c} $ a $
\theta $-stable Cartan subalgebra contained in $ \lie{b}_x. $ Then $
\lie{b}_x $ determines a set of positive roots $ R^+ $ in $ \lie{c}^*. $
Let $ (\lie{h}^*,\Sigma, \Sigma^+)\to (\lie{c}^*,R,R^+) $ be the
specialization map.  We can lift $ \theta $ along this map to obtain an
involution $ \theta_S $ which depends only on the orbit $ S $ (not the
particular basepoint).  From this, we have the following subsets of
roots:
\begin{align*}
    \Sigma_{S,im}&=\{\alpha\in \Sigma:  \theta_S(\alpha)=\alpha\} \\
    \Sigma_{S,\R}&=\{\alpha\in \Sigma:\theta_S(\alpha)=-\alpha\} \\
    \Sigma_{S,\C}&=\{\alpha\in \Sigma:  \theta_S(\alpha)\neq \pm \alpha\}.
\end{align*}
consisting of $ S $-imaginary, $ S $-real, and $ S $-complex roots
respectively.  Finally, we have the subset $ D_+(S)=\{ \alpha\in \Sigma_
{S,\C}^+:  \theta_S(\alpha)\in \Sigma^+_{S}\}. $ As shown in Section $ 5
$ of
\cite{Hecht2022}, $ D_+(S) $ controls the dimension of the orbit $ S $
attached to $ \lie{c}. $ That is, there exist orbits attached to $ \lie{c}
$ for which $ |D_+(S)| $ is minimal and maximal respectively.  We call
these orbits a \textbf{Langlands orbit} and a \textbf{Zuckerman orbit}
respectively.  Let $ S^L $ denote the Langlands orbit corresponding to $
\lie{c}. $

\begin{lemma}
    \cite[Section 5]{Hecht2022} The set of roots $ \Sigma^+_{S^L}\cup
    \Sigma_{S^L,im} $ is a parabolic root system in $ \Sigma. $ Let $
    \lie{p} $ be the parabolic subalgebra defined by this set of roots.
    If $ \lie{c} $ is stable under the involution defining the real
    form, the parabolic subalgebra $ \lie{p} $ is cuspidal.
\end{lemma}

We thus have a Levi decomposition \[ \lie{p}=\lie{l}\ds \lie{u}\]
with $\lie{l}$ corresponding to the $S^L$-imaginary roots. 
This is a reductive subalgebra. 

\begin{lemma}\label{Intersection_with_k}
	$\lie{p}\cap \lie{k}=\lie{l}\cap \lie{k}.$ Further, $\lie{l}\cap \lie{k}$ is spanned by $\lie{t}$ 
	and the compact imaginary root spaces. 
\end{lemma}   
\begin{proof} 
	Note that $\theta(\lie{l})=\lie{l}$ as each root imaginary root subspace is fixed. 
	On the Langlands orbit, $\theta$ sends all positive complex roots to negative roots. 
	Therefore, we have that $\theta(\lie{u})\cap \lie{u}=\{0\}.$ Let $A\in \lie{p}\cap \lie{k}.$
	 Then there exist elements $B\in \lie{l}$ and $C\in \lie{u}$ such that 
	 $ A=B+C.$ Consider \[ A=\theta(A)=\theta(B)+\theta(C).\]
	 As $\theta(C)\in \theta(\lie{u}),$ we must have that $\theta(C)=0.$ Hence, $A=B$ is 
	 contained in the Levi factor. 
	 
	 The second statement is evident. 
\end{proof} 

As we assume that $G_\R$ has a single conjugacy class of Cartan subgroups, there is 
a unique conjugacy class of cuspidal parabolic subgroups. Namely, these are all minimal parabolic 
subgroups and thus in a Langlands decomposition of $\lie{p}=\lie{m}\ds\lie{a}\ds\lie{u}$
 we have that $\lie{m}=\lie{k}\cap \lie{l}.$ Let $P$ and $M$ be the 
 corresponding subgroups of $G$ with Lie algebras $\lie{p}$ and $\lie{m}.$

By its construction $P$ contains $B_x$ and further:

\begin{corollary} 
	The stabilizer $S_x=K\cap B_x\subseteq M.$  
\end{corollary} 

By fixing a point in the open orbit, we also fix an identification $S^L\simeq K/S_x.$ Denote
 by $i:S^L\to \mathcal{B}$ the inclusion of the orbit. By \cite{Hecht1987}, this is an affinely 
 embedded subvariety. We want to compute the cohomology of the standard
  Harish-Chandra sheaf $\mathcal{I}(S^L,\tau)=i_+(\tau).$ As $i$ is an open 
  immersion, we have that  $i_+=i_*$ is the usual $\mathcal{O}_\mathcal{B}$-module pushforward. 
  Furthermore, \[ H^p(\mathcal{B},\mathcal{I}(S^L,\tau))\simeq H^p(S^L,\tau).\]
   Thus, our problem has reduced to computing these cohomology groups. 
   
   By the previous corollary, we have a projection \[ pr: S^L\to K/M\] with fiber 
   $M/(K\cap B_x)$ and by Lemma \ref{Intersection_with_k}, we have:
   
   \begin{lemma}\label{Fibers_of_pr}
   	$K\cap B_x$ is a Borel subgroup in $M.$ 
   \end{lemma}   
  Thus, the fibres of the projection $pr$ are exactly flag varieties for $\lie{m}.$  

Recall now the following fundamental fact from homological algebra. 

\begin{theorem}(Grothendieck--Leray--Serre) 
	Let $\mathcal{A},\mathcal{B},\mathcal{C}$ be abelian categories with enough
	 injectives and $F:\mathcal{A}\to \mathcal{B},$ $G:\mathcal{B}\to \mathcal{C}$ 
	 left exact functors. Suppose that $F$ carries injective objects to $G$-acyclic objects. 
	 Then there is a convergent spectral sequence with $E_2$ page 
	 \[ E_2^{pq}=R^pG\comp R^qF\implies R^{p+q}(G\comp F). \]
\end{theorem} 

Recall that a $K$-equivariant connection is a coherent $\mathcal{D}$-module which is also
 coherent as an $\mathcal{O}$-module. To compute the cohomology groups $H^p(S^L,\tau),$
  we will invoke the composition 
 \[\begin{tikzcd}
S^L \arrow[rd, "g"'] \arrow[r, "pr"] & K/M \arrow[d, "g'"] \\
                                     & \{*\}              
\end{tikzcd}\]
where $g$ and $g'$ are the unique maps to a point (realized as $\operatorname{Spec} \C$ for instance).  
Then $\Gamma(S^L,-)=g_*(-)$ (resp. $g'$ for K/M).Therefore, we have the 
spectral sequence \[ E_2^{pq}=H^q(K/M,R^p(pr)_*(\tau))\implies H^{p+q}(S^L,\tau). \]
As $K/M$ is affine (it is the quotient of a reductive group by a reductive subgroup
 \cite{Borel1991}), the higher cohomology of quasi-coherent sheaves vanishes. 
 Therefore, our spectral sequence degenerates on the second page and we
  obtain isomorphisms \[ H^{0}(K/M, R^p(pr)_*\tau)\simeq H^p(S^L,\tau). \]
The $R^p(pr)_*$ are $K$-equivariant and thus is the sheaf of sections of 
$K\times_M V$ for $V$ a representation of $M.$ $pr$ is proper and thus 
$V$ is necessarily finite dimensional. 

To compute $V$, it suffices to compute the stalks of $R^p(pr)_*\tau.$
 Let $X_y=(pr)^{-1}(y)$ for $y\in K/M.$ The stalks of the pushforward 
 sheaf are given 
 (by \cite[Proposition 8.1]{Hartshorne1977}) \[ (R^p(pr)_*\tau)_y\simeq H^p(X_y,\tau|_{X_y}). \]
By \ref{Fibers_of_pr}, we have that $X_y\simeq X_\lie{m}$ the flag variety of $\lie{m}.$ 
Therefore, the cohomology of $\tau|_{X_y}$ is completely determined by the Borel-Weil-Bott 
Theorem. Let $\rho_{im}$ denote the half-sum of the positive imaginary roots.  
\begin{proposition}
	Let $\lambda$ be the antidominant choice of infinitesimal character for a 
	finite dimensional representation of $M.$ Then $\lambda+\rho_{im}$ is 
	antidominant for the positive system determined by $\lie{b}_x$ on $\lie{t}.$
	 The cohomology groups $H^p(X_\lie{m},\tau|_{X_\lie{m}})$ vanish for $p>0.$
	  For $p=0,$ $H^0(  X_\lie{m},\tau|_{X_\lie{m}})$ is the irreducible 
	  representation of $M$ of infinitesimal character $\lambda.$ If $\lambda$
	   is singular then $H^p( X_\lie{m},\tau|_{X_\lie{m}})=0$ for all $p.$
\end{proposition} 
In addition to the statement of the Borel-Weil theorem, the only observation
 needed for the proof of this proposition is that $\rho|_\lie{t}=\rho_{im}|\lie{t}.$
  The geometric fiber of $\tau$ is defined by the specialization of $\lambda+\rho$
   and the restriction to $\lie{t}$ is precisely \[ s_x(\lambda+\rho_{im})|_\lie{t}.\] 
   This classifies $H^0( X_\lie{m},\tau|_{X_\lie{m}})$ as an $M$-module. We inflate
    this to a $(\lie{p},M)$-module by having $\lie{a}$ act by the specialization
     of $\lambda+\rho$ and $\lie{u}$ acting trivially. Denote $(\lie{p},M)$-module as $U.$   
  
We now want to identify the entire cohomology group $H^0(K/M,(pr)_*\tau).$ 
Let $s\in H^0(K/M, (pr)_*\tau)$ and consider the function
 \begin{align*} T_s&:U(\lie{g})\longrightarrow U \\ T_s(\xi)&=(\xi s)(pr(x)).\end{align*} 
 This map intertwines the $(\lie{p},M)$-module structures and thus gives us a map 
 $s\mapsto T_s$ from $H^0(K/M, (pr)_*\tau)$ to
  $\operatorname{coind}_{(\lie{p},M)}^{(\lie{g},M)}(U)=\Hom_{\lie{p},M}(U(\lie{g}),U).$ 
 It is injective trivially. Since $K$ acts algebraically on $ H^0(K/M, (pr)_*\tau)$, we 
 see that this descends to a $(\lie{g},K)$-morphism \[ H^0(K/M, (pr)_*\tau)\to I_{(\lie{p},M)}^{(\lie{g},K)}(U). \]
 These admissible $(\lie{g},K)$-modules have the same $K$-mulitiplicities
  by Frobenius reciprocity and thus this map is an isomorphism. 
 
 We summarize the result as follows: 
 
 \begin{theorem} 
 	Let $\lambda\in \lie{h}^*$ be anti-dominant for $\Sigma_{S^L,im}.$
	 Then \begin{enumerate} 
		\item If $\lambda$ is $\Sigma_{S^L,\im}$-regular,
		 $\Gamma(\mathcal{B},\mathcal{I}(S^L,\tau))\simeq  I_{(\lie{p},M)}^{(\lie{g},K)}(U).$
		\item If $\lambda$ is $\Sigma_{S^L,\im}$-singular, 
		then $ \Gamma(\mathcal{B},\mathcal{I}(S^L,\tau))=0.$ 
		\item $H^p(\mathcal{B},\mathcal{I}(S^L,\tau))=0$ for $p>0.$ 
	\end{enumerate} 	
 \end{theorem}  

Combining this with \cite[Theorem 11.47]{KnappVogan1995},
 we realize the cohomological induction functor as parabolic induction. 
\begin{corollary} 
	 If $\lambda$ is $\Sigma_{S^L,\im}$-regular, $\Gamma(\mathcal{B},\mathcal{I}(S^L,\tau))\simeq \operatorname{Ind}_P^G(U)_{[K]}.$ 
\end{corollary}

If $ S $ is any $ K $-orbit, then the Langlands orbit attached to
the same $ \theta $-stable Cartan subalgebra is the open orbit.  By transversality, we can
find an element $ w\in W(\lie{g},\lie{h}) $ such that the following
identity of functors exists $ \textbf{L}I_w\comp \textbf{R}\Gamma=\textbf
{R}\Gamma. $ This then implies that $ \Gamma(S,\tau)=\Gamma(S^L,\tau^L) $
and thus our results before apply with the slight modification $ \lambda\mapsto
w\lambda. $ The only result left to check is that the new character is still antidominant.  
\begin{lemma}
    \label{Antidominant_translation} Assume $ \lambda\in \lie{h}^* $ is
    strongly antidominant.  Let $ x\in S $ be a point in a $ K $-orbit
    on $ \mathcal{B} $ and $ \lie{c} $ a $ \theta $-stable Cartan
    subalgebra contained in $ \lie{b}_x. $ Let $ w $ be a Weyl group
    element transversal to $ S $ which intertwines $ S $ with the
    associated Langlands orbit attached to $ \lie{c}. $ Then $ s_y(w\lambda-\rho)|_\lie
    {t}+\rho_{im} $ is strongly antidominant for the positive root
    system in $ \lie{m} $ defined by $ \lie{f} $ for any $ y\in S^L $
    such that $ \lie{b}_y $ contains $ \lie{c}. $
\end{lemma}

\begin{proof}
    Let $ y\in S^L $ be a point in $ S^L $ such that $ \lie{b}_y $ also
    contains $ \lie{c}. $ Then the positive systems determined by $ \lie
    {b}_x $ and $ \lie{b}_y $ are related by $ \Sigma_S^+=w\Sigma_{S^L}^+.
    $ Now by construction $ s_y=s_x\comp w^{-1} $ and
    \[
        s_y(w\lambda-\rho)|_\lie{t}+\rho_{im}=s_x(\lambda-w^{-1}\rho)|_\lie
        {t}+\rho_{im}=s_x(\lambda)|_\lie{t}-w^{-1}s_x(\rho)|_\lie{t}+\rho_\lie
        {f}
    \]
    Therefore, to show that this weight is strongly antidominant, it
    suffices to check it on each summand.

    By our choice of points $ x $ and $ y, $ the weight $ -w^{-1}s_x(\rho)|_\lie
    {t}+\rho_{im} $ is the half sum of roots in $ w^{-1}\Sigma_{S}^-\cup
    \Sigma_{S^L,im}^+. $ By the choice of $ w, $ we have that $ w^{-1}\Sigma_S^+=\Sigma_
    {S^L}^+. $ Thus, $ -w^{-1}s_x(\rho)|_\lie{t}+\rho_{im} $ is a
    half sum of roots in $ \Sigma_{S^L}^-\cup \Sigma_{S^L,im}^+. $
    Therefore, if we set $ \Sigma_{S^L,RC}^+ $ to be the subset of
    positive roots which are not imaginary, then $ -w^{-1}s_x(\rho)+\rho_\lie
    {f} $ is the half sum of roots in $ \Sigma_{S^L,RC}^-. $ Therefore,
    for any $ \beta\in \Delta^+(\lie{m},\lie{t}), $ $ \beta^\vee(-w^{-1}s_x
    (\rho)+\rho_{im}) $ is both real and negative.  Hence, $ -w^{-1}s_x
    (\rho)+\rho_{im} $ is strongly antidominant.

    Now consider $ s_x(\lambda)|_\lie{t}. $ Put
    \[
        \Sigma_{S,w}^+=\{ \alpha\in \Sigma^+:  w\alpha\notin \Sigma_{S}^+\}.
    \]
    By
    \cite[Lemma 5.4]{Hecht2022}, this set (for our choice of $ w $) is
    contained in $ D_+(S) $ and consists entirely of positive complex
    roots.  We claim this observation is sufficient to obtain the strong
    antidominance of $ s_x(\lambda)_|\lie{t}. $ To see this, notice that
    if $ \beta\in \Delta^+(\lie{m},\lie{t})=\Sigma_{S^L,im}^+ $ then $ w\beta\in
    \Sigma_{S}^+ $ and $ \beta\in \Sigma_{S}^+ $ if and only if $ w\beta\notin
    \Sigma_{S,w^{-1}}^+. $ We identify this latter set with $ -w\Sigma_{S,w}^+.
    $ In particular, $ \beta\in \Sigma_{S}^+ $ if and only if $ -\beta\in
    \Sigma_{S,w}^+. $ However, as $ \beta $ is imaginary, $ -\beta\notin\Sigma_
    {S,w}^+\subseteq \Sigma_{S,\C}^+. $ Therefore, we have that $ \Sigma_
    {S^L,im}^+\subseteq \Sigma_{S}^+ $ and as $ \lambda $ is strongly
    antidominant, we see that
    \[
        \operatorname{Re}
        \beta^{\vee}(s_x(\lambda)|_\lie{t})\leq 0.
    \]
    This finishes the proof.
\end{proof}

\begin{corollary}
    Let $ G_\R=GL(n,\HH), K=Sp(2n,\C) $, and $ \mathcal{B}=Fl(\C^{2n}) $
    the full flag variety.  The minimal parabolic subgroups $ P_{min}=MAN
    $ have $ M $ given by $ SL(2,\C)^n $ and $ M_\R\simeq SU(2)^n. $
    Then the following hold:
    \begin{enumerate}
        \item
            $ \lie{h}\simeq \lie{t}\ds \lie{t} $ for $ \lie{t} $ the Lie
            algebra of a maximal torus in $ K. $
        \item
            \cite{RichardsonSpringer1990} The $ K $-orbits in $ \mathcal
            {B} $ are in bijection with fixed-point free involutions.
        \item
            Assume $ \lambda$ is the infinitesimal
            character of a finite dimensional representation of $ K $
            with $ \lambda $ the lowest weight.  Then for each orbit $ S
            $, there is a unique irreducible $ K $-equivariant
            connection $ \mathbbm{1}_S $ on $ S $ compatible with $
            \lambda+\rho. $
        \item
            Fix $ x\in S, y\in S^L $ and let $ \lie{c} $ be as above.
            Let $ \tau^L $ be the unique $ K $-equivariant connection on
            $ S^L $ compatible with $ w(\lambda-\rho)+\rho. $ Set $ E=\mathbbm
            {1}_{S^L}^\vee\tensor T_x\omega_\mathcal{B}. $ Then $ E $
            has $ \lie{c} $-weight $ s_y(\rho-w(\lambda)). $
        \item
            The corresponding finite dimensional representation of $ M_\R
            $ given by induction in stages is precisely the unique
            finite dimensional representation with lowest $ \lie{t} $-weight
            given by the restriction of $ s_y(\rho-w(\lambda-\rho_{im})).
            $
        \item
            \cite{Primc1979} The highest weight of the minimal $ K $-type
            of a tempiric representation of $ G_\R $ is given by the
            unique dominant $ W_K $-conjugate of the highest weight of
            the representation of $ M. $
    \end{enumerate}
\end{corollary}

\begin{proof}
    Parts (1),(3),(4), and (5) are apparent from the calculations above.  (2) can
    be found in the final example of
    \cite{RichardsonSpringer1990}.  (6) is Theorem 2 of Section 15 of
    \cite{Primc1979}.  
\end{proof}


\section{The Lusztig--Vogan Bijection}

We are finally capable of computing the Lusztig-Vogan bijection explicitly in the cases of $GL(2,\HH)$ and $GL(3,\HH).$ These give the first examples of real groups that have a Lusztig--Vogan bijection which are not already complex groups.  

\subsection{The Case of $ GL(2,\HH) $}

We now want to give explicit
computations of the above results for $ GL(2,\HH) $ and show how they
exhibit the Lusztig--Vogan bijection.  Let us start with a bit of some
explanation as to how the bijection should work in general.  We want to
find a way to assign to every pair $ (\OO_K,\tau_K) $ a unique
irreducible representation of $ K. $ Equivalently, we want to assign to
each pair a unique tempiric representation.  Morally, this should be
something of a leading term.  One key feature of the proposed bijection
is that the stabilizer of the zero orbit is always the entire group.
Thus, in order to fit in the rest of the orbit pairs, we are required to
shift the parameters coming from each orbit similar to Example 1.
We will show that this procedure works for $ GL(2,\HH) $ explicitly.

For $GL(n,\HH),$ we fix once and for all a choice of symplectic form whose 
matrix is anti-diagonal \[ J=\operatorname{antidiag}(-1,...,-1,1,...,1)\]
with $n$ copies of $-1$ and $1.$ The Cartan decomposition of its Lie algebra
is given as \[ \lie{gl}(n,\HH)=\lie{sp}(n)\ds \lie{s} \]
with matrices in $\lie{sp}(n)$ having the form \[ X=\Mat{ A& B \\ C & -A^{at}}\]
where $A^{at}$ is the ``anti-transpose" of $A$ with entries $A^{at}_{ij}=(a_{n+1-j,n+1-i}),$ 
and where $B=B^{at}$ and $C=C^{at}.$ Matrices in $\lie{s}$ have the form \[ S=\Mat{A & B \\ C & A^{at}}\]
withe $B=-B^{at}$ and $C=-C^{at}.$

Begin with the zero orbit $ \OO_0. $ In this case the stabilizer $ K^0=K
$ is the entirety of $ K. $ Thus, the Zuckerman Character formula for
the trivial representation gives the shift!  The character formula in
this case is best given by the Cousin resolution.  To write this down,
we need to know the $ K $-orbits on $ G/B. $ We have a parametrization
of them
\cite{RichardsonSpringer1990} given by fixed-point free involutions of $
4. $ The closure order is then the reverse Bruhat order (essentially the
involutions which swap integers further away are lower down).  The
Cousin resolution of $ \mathcal{D}_{-\rho} $-modules then takes the
following form:
\[
    [\C]=\mathcal{I}((12)(34),\mathbbm{1})-\mathcal{I}((13)(24),\mathbbm
    {1})+\mathcal{I}((14)(23),\mathbbm{1}).
\]
As $ GL(n,\HH) $ has a single conjugacy class of Cartan subalgebras (hence
subgroups), we know that the standard $ \mathcal{D} $-module supported
on the closed orbit $ (14)(23) $ is given as the cohomologically induced
module attached to a $ \theta $-stable Borel subalgebra.  Thus
\[
    \mathcal{I}((14)(23),\mathbbm{1})=A_\lie{b}(0).
\]
This has lowest $ K $-type $ 2\rho(\lie{n}\cap\lie{s}). $ The weights of
$ \lie{t} $ on $ \lie{s} $ are given by the short roots of the type $ C_2
$ roots system.  Thus, in the standard basis (see \cite{Knapp2005} Chapter 2), 
\[
    2\rho(\lie{n}\cap \lie{s})=(2,0).
\]
The lowest $ K $-type of the other standard modules are given by
combining the results of
\cite{Primc1979} and Theorem \ref{Global_Sections_from_D}.  They are
precisely $ (0,0) $ and $ (1,1) $ respectively.  Thus, we see that the
largest lowest $ K $-type appearing is coming from the closed $ K $-orbit
on $ G/B. $

Turning our attention to the principal $K$-orbit $ \OO_{prin}, $ we have
that $ K^{prin}_{red}=SL(2,\C) $ and $ L\cap K=GL(2,\C) $ and $ L=GL(2,\C)^2.
$ As was done in Example \ref{K_as_Tempiric_Example}, we have that the
Weyl Character formula gives $ L\cap K $ representations as alternating
sums of tempiric representations.  Let $ n\geq 1 $ be the infinitesimal
character of an irreducible finite dimensional representation of $ K^{prin}_
{red}. $ We can realize this as the restriction of an irreducible
representation of $ L\cap K $ with infinitesimal character
\[
    \left(\left\lceil \frac{n}{2} \right \rceil,-\left\lfloor \frac{n}{2}
    \right \rfloor \right).
\]
As the determinant character of $ L\cap K $ restricts trivially to $ SL(2,\C),
$ we could have equivalently picked any irreducible representation with
infinitesimal character of the form
\[
    \left(\left\lceil \frac{n}{2} \right \rceil +k,-\left\lfloor \frac{n}
    {2} \right \rfloor +k \right)
\]
to reflect the appearance of $ \det^{\tensor k}. $ Now the Weyl
character formula for $GL(2,\C)$ as above, gives
\begin{align*}
    V_{\left(\left\lceil \frac{n}{2} \right \rceil +k,-\left\lfloor
    \frac{n}{2} \right \rfloor +k \right)}&= X\left(\left(\frac{\left\lceil
    \frac{n}{2} \right \rceil +k}{2},\frac{-\left\lfloor \frac{n}{2}
    \right \rfloor +k}{2}\right),\left(\frac{\left\lceil \frac{n}{2}
    \right \rceil +k}{2},\frac{-\left\lfloor \frac{n}{2} \right \rfloor
    +k}{2}\right)\right)\\
    &-X\left(\left(\frac{\left\lceil \frac{n}{2} \right \rceil +k+1}{2},\frac
    {-\left\lfloor \frac{n}{2} \right \rfloor +k-1}{2}\right),\left(\frac
    {\left\lceil \frac{n}{2} \right \rceil +k+1}{2},\frac{-\left\lfloor
    \frac{n}{2} \right \rfloor +k-1}{2}\right)\right ).
\end{align*}

Pushing-forward along $ \mu $ is given by cohomological induction as
above.  As we have the version of the Weyl Character formula given as
standard modules, we know how cohomological induction behaves with
respect to these modules:  simply add $ \rho(\lie{u}) $ to the
parameter;  $ \rho(\lie{u}) $ in this case is $ (1,1,-1,-1). $ Thus, the
Euler characteristic of the pushforward:  $ \chi(\textbf{R}\mu_*\pi^*V_{\left
(\left\lceil \frac{n}{2} \right \rceil +k,-\left\lfloor \frac{n}{2}
\right \rfloor +k \right)}) $ is given by
\begin{align*}
    &X\left(\left(\frac{\left\lceil \frac{n}{2} \right \rceil +k+2}{2},\frac
    {-\left\lfloor \frac{n}{2} \right \rfloor +k+2}{2}\right),\left(\frac
    {\left\lceil \frac{n}{2} \right \rceil +k-2}{2},\frac{-\left\lfloor
    \frac{n}{2} \right \rfloor +k-2}{2}\right)\right) \\
    &-X\left(\left(\frac{\left\lceil \frac{n}{2} \right \rceil +k+3}{2},\frac
    {-\left\lfloor \frac{n}{2} \right \rfloor +k+1}{2}\right),\left(\frac
    {\left\lceil \frac{n}{2} \right \rceil +k-1}{2},\frac{-\left\lfloor
    \frac{n}{2} \right \rfloor +k-3}{2}\right)\right ).
\end{align*}

We now wish to extract $ K $-type information from the above standard
modules.  To do so, we need to translate back to the language of induced
representations.  The method of doing this mirrors the complex group
world as $ GL(n,\HH) $ has a single conjugacy class of Cartan subgroups.
We take the standard module parameter $ ((a,b),(c,d)) $ and add them.
This gives us $ (a-c,b-d) $ as the infinitesimal character of the $ SU(2)^2
$ representation.  If this is dominant, then we can obtain the highest
weight by subtracting $ (1,1) $ ($ \rho $ for $ SU(2)^2 $).  So in
general, we first dominate the character (using the Weyl group of $ K $)
and then subtract $ \rho. $

Thus, we see that the infinitesimal characters of the $ SU(2)^2 $
representations corresponding to the standard modules above are as
follows:
\begin{align*}
    \left(\left\lceil \frac{n}{2} \right \rceil +k+2,-\left\lfloor \frac
    {n}{2} \right \rfloor +k+2 \right)&& \left(\left\lceil \frac{n}{2}
    \right \rceil +3+k,-\left\lfloor \frac{n}{2} \right \rfloor+1+k
    \right).
\end{align*}
We want to minimize the norm of these following Vogan's philosophy.  To
determine this, we need to $ K $-dominate the characters first.  The
resulting infinitesimal characters are
\begin{align*}
    \left(\left\lceil \frac{n}{2} \right \rceil +k+2 ,\left\lfloor \frac
    {n}{2} \right \rfloor -k-2 \right)&& \left(\left\lceil \frac{n}{2}
    \right \rceil+k +3,\left\lfloor \frac{n}{2} \right \rfloor-1-k
    \right).
\end{align*}
For small positive $ n, $ this negation of the second factor may give
the negative of the desired integer.  However, as can be seen for $ n=0,
$ these infinitesimal characters are minimized when $ k=-2. $

Therefore, we get
\begin{align*}
    \left(\left\lceil \frac{n}{2} \right \rceil ,\left\lfloor \frac{n}{2}
    \right \rfloor \right)&& \left(\left\lceil \frac{n}{2} \right \rceil
    +1,\left\lfloor \frac{n}{2} \right \rfloor+1 \right).
\end{align*}
Subtracting $ (1,1) $ from both we obtain the highest weight of the
lowest $ K $-types:
\begin{align*}
    \left(\left\lceil \frac{n}{2} \right \rceil-1 ,\left\lfloor \frac{n}
    {2} \right \rfloor-1 \right)&& \left(\left\lceil \frac{n}{2} \right
    \rceil ,\left\lfloor \frac{n}{2} \right \rfloor \right).
\end{align*}
Therefore we see that the larger of the two $ K $-types appearing here
is $ \left(\left\lceil \frac{n}{2} \right \rceil ,\left\lfloor \frac{n}{2}
\right \rfloor \right). $

This gives the following result:

\begin{theorem}
    Let $ G=GL(2,\HH). $ There is a bijection $ \gamma:\{(\OO,\tau)\}\to
    \Lambda^+_K $ given as follows
    \begin{align*}
        (\OO_0,(a,b))&\mapsto (a+2,b) \\
        (\OO_{prin},n)&\mapsto
        \begin{cases}
            (k,k) & n=2k \\
            (k+1,k) & n=2k+1
        \end{cases}.
    \end{align*}
\end{theorem}


\subsection{The Principal and Zero Orbit Generally}

We are now equipped with the necessary tools to compute the $LKT$
 map on the zero and principal orbits for general $n.$ 
These calculations mirror the $GL(2,\HH)$ case and 
thus most of the details follow from that calculation. 


\subsubsection{The Zero Orbit}

 The general shift will mimic the case of $
GL(2,\HH). $ The largest shift should appear on the zero orbit and comes
from the short roots in the type $ C_n $ root system.  As above, $ K^0=K
$ and thus the resolution by tempiric representations is given by taking
the Zuckerman character formula and setting the continuous parameter to
zero.  By Corollary \ref{Cousin_Resolution_Corollary} and Proposition
\ref{Global_Sections_from_D}, we see that the term with the largest
lowest $ K $-type is supported on the closed orbit and is given by $ A_\lie
{b}(0). $ The lowest $ K $-type of this module is
\[
    2\rho(\lie{n}\cap \lie{s})=(2(n-1),2(n-2),...,2,0).
\]
Therefore, we send
\[
    (\OO_0, (a_1,...,a_n))\mapsto (a_1+2(n-1),a_2+2(n-2),...,a_n).
\]


\subsubsection{The Principal Orbit}

We now deal to the other extreme.
In this case, the associated partition is $ [n,n]. $ The stabilizer is $
G^{[n,n]}_{red}=GL(2,\C) $ embedded diagonally in $ L=GL(2,\C)^n. $ The $
K $-stabilizer is $ K^{[n,n]}_{red}=SL(2,\C) $ embedded diagonally as
well.  To be a bit more consistent with other cases, we shall instead
say $ K^{[n,n]}_{red}=Sp(2,\C). $ This lives inside
\[
    L\cap K=
    \begin{cases}
        GL(2,\C)^{\frac{n}{2}} & n \text{ even} \\
        GL(2,\C)^{\frac{n-1}{2}}\times Sp(2,\C) & n \text{ odd}
    \end{cases}.
\]
Therefore, lifting of a representation of the stabilizer to a
representation of $ L\cap K $ amounts the $ GL(2,\HH) $-case above done
either $ \frac{n}{2} $ times (in the even case) or $ \frac{n-1}{2} $
times (in the odd case). Concretely, irreducible representations of $ Sp(2,\C)=SL(2,\C)
$ are parametrized by positive integers given by their highest weight.
So starting with a positive integer $ \tau, $ we lift it to an
irreducible $ L\cap K $ module $ \tilde{\tau} $ in the following two
ways:

\[
    \tilde{\tau}=
    \begin{cases}
        \left( \left\lceil \frac{\tau}{m} \right \rceil+k_1, -\left\lceil
        \frac{\tau-1}{m} \right \rceil+k_1 ,...,\left\lceil \frac{\tau-(m-2)}
        {m} \right \rceil +k_{m},-\left\lceil \frac{\tau-(m-1)}{m}
        \right \rceil +k_{m} \right) & n=2m \\
        & \\
        \left( \left\lceil \frac{\tau}{n} \right \rceil+k_1,- \left\lceil
        \frac{\tau-1}{n} \right \rceil+k_1 ,...,-\left\lceil \frac{\tau-
        (n-2)}{n} \right \rceil +k_{m},\left\lceil \frac{\tau-(n-1)}{n}
        \right \rceil \right) & n=2m-1
    \end{cases}
\]
with the $ k_i \in \Z $ corresponding to different powers of the
determinant.  The lack of a shift in the final term in the odd case
follows from the trivial lifting on this factor of $ Sp(2,\C). $ These
are highest-weights for some irreducible representation of $ L_K=L\cap
K. $

We are now tasked with writing this representation of $ L\cap K $ in
terms of tempiric $ (\lie{l},L\cap K) $-modules.  The easier of the two
cases is when $ n=2m $ is even.  In this case, $ L\cap K $ is isomorphic
to a real form of $ L $ and thus we can treat this situation as in the
complex group situation (see Example \ref{K_as_Tempiric_Example}).  The
Weyl group of $ L\cap K $ here is $ (\Z/2\Z)^{n/2} $ and thus we see the
longest element is $ w_0=(-1,...,-1). $ Therefore, by the Weyl character
formula, we obtain
\[
    [\tilde{\tau}]= \sum_{w\in W_{L_K}-\{w_0\}} (-1)^{l(w)} \Ind_{B}^G(
    \tilde{\tau}+\rho_{L_K}-w\rho_{L_K})+ (-1)^{n/2} \Ind_B^G(\tilde{\tau}+2\rho_
    {L_K}),
\]
the final term being the contribution of $ w_0. $ Now that we have an
expression in terms of tempiric representations for $ L, $ we want to
translate these into the form coming from pseudocharacters.  Similar to
the $ GL(2,\C) $ case above, we take $ \Ind_B^G(\lambda)=X(\lambda) $
where we are only considering the restriction to $ T. $

Then, we obtain the following
\[
    \textbf{R}\mu_*\pi^*(W_{\tilde{\tau}})=\mathcal{R}_{\lie{q}}\left(\sum_
    {w\in W_{L_K}} X\left(\tilde{\tau}+\rho_{L_K}-w\rho_{L_K}\right)\right).
\]
Cohomological induction here results in addition to the infinitesimal
character by $ \rho(\lie{u}). $ Thus,
\[
    \textbf{R}\mu_*\pi^*(W_{\tilde{\tau}})=\sum_{w\in W_{L_K}} X\left(\tilde
    {\tau}+\rho_{L_K}-w\rho_{L_K}+\rho(\lie{u}) \right).
\]
As $ n=2m $ is even, $ \rho(\lie{u})|_T $ is easily calculated to be
\[
    \rho(\lie{u})|_T=(2n-2,2n-2,2n-6,2n-6,...,2,2).
\]
Finally, to obtain the lowest $ K $-type information, we add the
resulting pairs and take the dominant conjugate.  It is clear that the
largest lowest $ K $-type appearing on the right hand side is given by
the lowest $ K $-type of the standard module attached to $ w_0. $ Thus,
the the dominant conjugate of the largest infinitesimal character appearing on the right side is
\begin{align*}
    \operatorname{dom}
    _K\left( \tilde{\tau}+2\rho_{L_K}+\rho(\lie{u})|_T \right).
\end{align*}
Now, $ \tilde{\tau} $ can be chosen to minimize this largest $ K $-type
appearing in the extension.  In particular, let $ k_i=4(i-1)-2(n-1). $
Then the dominant conjugate of the infinitesimal character is
\begin{align*}
    \operatorname{dom}
    _K\left( \tilde{\tau}+2\rho_{L_K}+\rho(\lie{u})|_T \right)&=\\
    &\left( \left\lceil \frac{\tau}{m} \right \rceil+1, \left\lceil
    \frac{\tau-1}{m} \right \rceil+1 ,...,\left\lceil \frac{\tau-(m-1)}{m}
    \right \rceil +1 \right).
\end{align*}
Hence, the highest weight of this representation is
\begin{align*}
    \left( \left\lceil \frac{\tau}{m} \right \rceil, \left\lceil \frac{\tau-1}
    {m} \right \rceil ,...,\left\lceil \frac{\tau-(m-2)}{m} \right
    \rceil ,\left\lceil \frac{\tau-(m-1)}{m} \right \rceil \right).
\end{align*}

Thus for $ n $ even, we have the following map
\begin{equation}
    (\OO_{prin},\tau)\mapsto \left( \left\lceil \frac{\tau}{n} \right
    \rceil, \left\lceil \frac{\tau-1}{n} \right \rceil ,...,\left\lceil
    \frac{\tau-(n-2)}{n} \right \rceil ,\left\lceil \frac{\tau-(n-1)}{n}
    \right \rceil \right)
\end{equation}

We now treat the case of $ n $ odd.  As $ GL(n,\HH) $ has a single
conjugacy class of Cartan subgroups, we have that every $ \theta $-stable
Cartan in $ G $ is fundamental and thus $ \lie{h}^\theta=\lie{t} $ is a
Cartan subalgebra in $ \lie{k}. $ This gives the following reduction of
the pseudocharacters from before:

\begin{corollary}
    \cite[Theorem 1.2]{Vogan2007} Let $ G_\R $ be a real reductive
    algebraic group with a single conjugacy class of Cartan subgroups.
    Then the following sets are in natural bijection:
    \begin{enumerate}
        \item
            Irreducible representations of $ K. $
        \item
            Irreducible tempered representations of $ G_\R $ with real
            infinitesimal character.
        \item
            K-conjugacy classes of pseudocharacters $ (H,\Psi, \bar{\gamma})
            $ satisfying
            \begin{enumerate}
                \item
                    $ \Psi $ is the unique (up to conjugacy) set of
                    roots of $ L_\R=M_\R A_\R $ in the Langlands
                    decomposition of a minimal parabolic $ P_\R=M_\R A_\R
                    N_\R. $
                \item
                    $ \bar{\gamma}|_\lie{a}=0. $
                \item
                    $ \bar{\gamma}-\rho(\Psi)\in \mathbb{X}^*(T) $
                \item
                    $ \ip{\alpha,\gamma}>0 $ for all $ \alpha\in \Psi. $
            \end{enumerate}
    \end{enumerate}
\end{corollary}

\begin{remark}
    Notice that as $ \Psi $ determines $ P_\R, $ we can simplify the
    last two points to say that $ \gamma $ is $ L_\R $-dominant.  For (iii),
    we abuse the inclusion of characters of the torus by differentiation
    into $ \lie{t}^*. $
\end{remark}

\begin{proof}
    The only new content from Theorem \ref{Tempiric_LKT_Bijection} is $
    (b)\iff (c). $ The if and only if will follow from Theorem \ref{Langlands_Classification_Pseudo_Chars}
    once we classify the pseudocharacters.  Let $ (H,\Psi,\bar{\gamma}) $
    be a pseudocharacter for $ G_\R. $ As discussed before, this is
    immediately final as $ G_\R $ has a single conjugacy class of Cartan
    subalgebras.  A discrete pseudocharacter is one which is defined to
    be zero on the split part of $ \lie{h}. $ In the case of a complex
    group, we pass from a standard pseudocharacter to a discrete one by
    forgetting the restriction to $ \lie{a}. $ This is possible as $
    \lie{t} $ is a Cartan of $ K $ (in general it may be smaller).  This
    case is the same:  decompose $ \lie{h}=\lie{t}\ds \lie{a} $ in a
    Cartan decomposition.  $ \lie{t} $ is a Cartan subalgebra of $ \lie{k}.
    $ Hence, we obtain a discrete parameter by restriction.  Unfurling
    the definition then gives the characterization in the theorem.
\end{proof}

In addition to this reduction, we also can simplify the final theorem of
\cite{Vogan2007}.

\begin{corollary}
    \cite[Theorem 16.6]{Vogan2007}%
    \label{K_as_Tempirics} Let $ G_\R $ be a real reductive algebraic
    group with a single conjugacy class of Cartan subgroups.  Let $ \pi $
    be an irreducible representation of $ K. $ Let $ \gamma=(\Psi,
    \Gamma, \bar{\gamma}) $ be the associated limit character with $ LKT
    (X(\gamma))=\pi. $ Choose a strongly $ \gamma $-compatible parabolic
    subalgebra $ \lie{q}=\lie{l}\ds \lie{u} $ (in the sense of 11.6 in
    Loc.  Cit.).  By this choice, there is a one dimensional
    representation of $ L\cap K $, $ \pi_\lie{q} $, which is the lowest $
    L\cap K $-type of the principal series $ X_L(\gamma_\lie{q}). $ Then
    \[
        \pi=\sum_{A\subseteq \Delta(\lie{u}^{op}\cap \lie{s})} (-1)^{|A|}X_G
        (
        \operatorname{Cind}
        (\Psi, \pi_\lie{q})+2\rho(A))|_K
    \]
    with $\operatorname{Cind}$ a shorthand for the procedure of Proposition \ref{cohomological_induction_pseudo_chars}. 
\end{corollary}
\begin{proof}
    Vogan's result is written in terms of a triple sum.  The first sum
    is over conjugacy classes of Cartan subgroups and thus becomes a
    single term.  The second sum is over positive roots of $ L. $ As $ L
    $ is chosen to have only real roots and $ G_\R $ has a single
    conjugacy class of Cartan subgroups, $ L $ is a Cartan subgroup
    itself.  Thus, the theorem reduces to the above.
\end{proof}

We can now return to the $ n $ odd case.  In this situations $ L\cap K $
is a product of copies of $ GL(2,\C) $ and one $ Sp(2,\C). $ The copies
of $ GL(2,\C) $ arise from pairs in $ L $ and thus reduce to the complex
group case.  The $ Sp(2,\C) $ is where we need to invoke the above
corollary.  The real form we have in this situation is $ G_\R=\HH^\times
$ and $ K_0=Sp(1)=SU(2). $ We use the identity $ \HH^\times=SU(2)\times
\R^\times $ with $ \R^+ $ contained in the center of $ \HH^\times $.  In
particular, the decomposition tells us that irreducible representations $
\tau $ of $ \HH^\times $ are all finite dimensional and have a
decomposition as
\[
    \tau=\pi \boxtimes e^{ \nu x}.
\]
For them to be tempiric, $ \nu=0 $ and thus the tempiric representations
are simply irreducible representations of $ SU(2). $ In terms of
pseudocharacters, we only need to keep track of the infinitesimal
character.  So, let $ \tilde{\tau} $ be as above.  The $ Sp(2,\C) $
factor gives a tempiric with infinitesimal character $ 1 $ (when
restricted to $ T $).  On the other factors, we have the Weyl Character
formula which by the general considerations in the even case tells us to
set $ k_i=2+4i-2n. $ This gives us a temperic with lowest $ K $-type
\[
    \left( \left\lceil \frac{\tau}{n} \right \rceil, \left\lceil \frac{\tau-1}
    {n} \right \rceil ,...,\left\lceil \frac{\tau-(n-2)}{n} \right
    \rceil ,\left\lceil \frac{\tau-(n-1)}{n} \right \rceil \right).
\]

Notice that even though we needed an additional step in writing $ \tilde
{\tau} $ as tempiric $ (\lie{l},L\cap K) $-modules, we still arrived at
a lowest $ K $-type which aligns with that of the even case.  This
should not be surprising given the general construction above.

In sum, we have the following: 

\begin{theorem} 
	The $LKT$ map for the principal orbit of $GL(n,\HH)$ is given by \[ (\mathbb{O}_{prin},\tau)\mapsto \left( \left\lceil \frac{\tau}{n} \right \rceil, \left\lceil \frac{\tau-1}
    {n} \right \rceil ,...,\left\lceil \frac{\tau-(n-2)}{n} \right
    \rceil ,\left\lceil \frac{\tau-(n-1)}{n} \right \rceil \right).\]
\end{theorem} 


\subsection{The Case of $ GL(3,\HH) $}

The case of $ GL(3,\HH) $ adds a new challenge to the computation of the
$ LKT $ map.  There are three orbits:  the zero orbit, the principal
orbit, and the middle orbit.  The final of which corresponds to the
partition $ [2,2,1,1] $ of $ 6. $ We shall call this $ \OO_{mid}. $ This is
our first encounter with an odd orbit!  Recall that this means that $
\lie{s}^2\neq \lie{u}\cap \lie{s} $ for the Jacobson-Morozov parabolic.
In particular, we have a non-trivial weight one eigenspace of $ H, $ $
\lie{s}_1. $

To deal with this at a computational level, we need to modify our
formula from before:
\[
    \sum_{j}(-1)^j
    \operatorname{gr}
    \mathcal{R}_\lie{q}\left(\sum_{\pi} M(\pi,\eta) \overline{X}(\gamma_\pi)\tensor
    \textstyle{\bigwedge^j}\lie{s}_1^*\right).
\]
Let $ \Delta(\lie{s}_1,T) $ denote the $ T $-weights on $ \lie{s}_1. $ A
version of the Weyl Character formula tells us the following:

\begin{lemma}
    \cite[Corollary 12.2(d)]{AdamsVogan2021} Let $ V $ be a finite
    dimensional representation of $ T. $ Then as a virtual $ T $-module:
    \[
        \sum_{i=0}^{\dim V} (-1)^i \textstyle{\bigwedge^i} V=\displaystyle\sum_
        {A\subseteq \Delta(V,T)} (-1)^{|A|} 2\rho(A)
    \]
    where $ 2\rho(A) $ means the sum of the weights in $ A. $
\end{lemma}

From this, we can make the above formula computable:
\begin{equation}[A_\eta]
    =\sum_{A\subseteq \Delta(\lie{s}_1,T)} \sum_{\pi} M(\pi,\eta)
    \operatorname{gr}
    \mathcal{R}_\lie{q} \overline{X}(\gamma_\pi-2\rho(A))
\end{equation}
The $ M(\pi,\eta) $ are determined completely by Corollary \ref{K_as_Tempirics}.

We are now equipped with all of the necessary tools to compute the $ LKT
$ map on the middle orbit.  We need to establish the relevant structure
theory.  Let us conveniently pick our nilpotent element, and thus our $
\lie{sl}_2(\C) $-triple to be
\begin{align*}
    X=\Mat{ 0&0 &0 &0 & 1 & 0\\
    0& 0& 0& 0& 0& -1 \\
    0& 0& 0& 0& 0&0 \\
    0& 0& 0& 0& 0& 0\\
    0& 0& 0& 0& 0& 0\\
    0& 0& 0& 0& 0&0 } && H=\Mat{1 & 0 &0 &0 &0 &0 \\
    0 & 1 &0 &0 &0 &0\\
    0& 0& 0 &0 &0 &0 \\
    0& 0& 0& 0 &0 &0 \\
    0& 0& 0& 0& -1 &0 \\
    0 &0 &0 &0 &0 & -1}, && Y=X^T.
\end{align*}

This realizes the Levi-factor of the Jacobson-Morozov parabolic as $ L\simeq
GL(2,\C)^3 $ with each $ GL(2) $ embedded along the main diagonal.  This
also gives
\[
    L\cap K\simeq GL(2,\C)\times SL(2,\C)
\]
realized as block-diagonal matrices of the form $
\operatorname{diag}
(A,B,(A^{at})^{-1}). $ A quick calculation shows that $ K^X_{red}\simeq
SL(2,\C)\times SL(2,\C) $.  As $ SL(2,\C) $ is the semisimple part of $
GL(2,\C), $ any irreducible representation of $ GL(2,\C) $ restricts to
an irreducible representation of $ SL(2,\C). $ Therefore, it suffices to
determine which irreducible representations of $ GL(2,\C) $ restrict to
the trivial representation (these are the basis elements of the $
\mathbb{K} $-group).

The restriction map $ \mathbb{K}(L\cap K)\to \mathbb{K}(K^X_{red}) $ is
given on weights of irreducible representations by
\[
    ((a,b); n)\mapsto (a-b;n).
\]
Therefore, the fiber over the trivial representation (highest weight $ (0;0)
$) consists entirely of (sums of) weights $ ((k,k);0) $ corresponding to the character $ \det^{\tensor
k}\boxtimes \mathbbm{1} $.

Now, let $ \tau=(n_1;n_2) $ be the highest weight of an irreducible $ K^X_
{red} $-module.  This has a lift to $ L\cap K $ (and thus to $ P\cap K $)
as
\[
    \tilde{\tau}=\left( \left\lceil \frac{n_1}{2} \right \rceil+k, -\left\lfloor
    \frac{n_1}{2} \right \rfloor+k; n_2\right).
\]
We now mimic the approach from above:  write $ \tilde{\tau} $ in terms
of tempiric $ (\lie{l},L\cap K) $-modules.  For the $ GL(2,\C) $-part,
we have already seen how to do this.  For the $ SL(2,\C) $-part, we saw
how to deal with this in the principal orbit calculation.  In the end,
we obtain the following:
\[
    \tilde{\tau}=X\left( \left\lceil \frac{n_1}{2} \right \rceil+k, -\left\lfloor
    \frac{n_1}{2} \right \rfloor+k; n_2+1\right)- X\left( \left\lceil
    \frac{n_1}{2} \right \rceil+k+1, -\left\lfloor \frac{n_1}{2} \right
    \rfloor+k-1; n_2+1\right).
\]
where the notation of $ X(a,b;c) $ means the tempiric $ \Ind_B^{GL(2,\C)}
(a,b)\boxtimes V_{c} $ with $ V_{c} $ the finite dimensional
representation of $ SL(2,\C) $ of highest weight $ c $ and infinitesimal
character $ c+1. $ As we saw above, we now apply cohomological
induction.  As our group only has a single conjugacy class of Cartan
subgroups, we only need to keep track of the infinitesimal character.
The affects of cohomological induction are thus adding $ \rho(u)|_T. $
In this case,
\[
    \rho(u)|_T=(4,4,0).
\]

Before we rush forward in cavalier style, we need to consider $ \lie{s}_1.
$ In consists of matrices of the form
\[
    \lie{s}_1=\left \{ \Mat{ 0 &0 & a & c & 0 & 0\\
    0 & 0 & b & d & 0 & 0\\
    0 & 0 & 0 &0 & -d & -c \\
    0 & 0 & 0&0&-b&-a\\
    0&0&0&0&0&0 \\
    0&0&0&0&0&0}:  a,b,c,d\in \C \right\}.
\]
Therefore, the $ T $-weights of $ \lie{s}_1 $ are $ e_1-e_3,e_2-e_3,e_1+e_3,
$ and $ e_2+e_3. $ As there are four weights (roots), our sum above
becomes a $ 32 $-term monster.  Luckily, as we saw before, we only need
to consider the $ 16 $-terms coming with the $ +2\rho_{GL(2)} $ shift in
the Weyl Character formula.

At the trivial representation of $ K^X_{red}, $ we see that the affect
of the subset $ A=\{e_1-e_3,e_2-e_3\} $ yields the largest infinitesimal
character.  When $ k=-3 $ this infinitesimal character is minimized.  In
this situation, it is the unique largest infinitesimal character
appearing among the tempiric representations in the expansion.  If we
inspect the lowest $ K $-type formula for this tempiric, we see that we
obtain
\[
    \left(n_2+2, \left\lceil \frac{n_1}{2} \right \rceil, \left\lfloor
    \frac{n_1}{2} \right \rfloor\right).
\]
For this to be a $ K $-dominant $ T $-weight, we need $ n_1\leq 2n_2+1. $
What happens when we are not in this regime?

Considering the case of $ (2;0), $ will guide our expedition.  For this,
we get that a new tempiric has the largest infinitesimal character:  the
tempiric corresponding to the subset $ B=\{e_2-e_3,e_2+e_3\}. $ Up to
conjugation this has the unique largest infinitesimal character and
gives a $ K $-type formula of
\[
    \left( \left\lceil \frac{n_1}{2} \right \rceil+1, \left\lfloor \frac
    {n_1}{2} \right \rfloor+1,n_2\right).
\]
Additionally, we see that this is $ K $-dominant when $ n_1>2n_2+1$
and the assignment \[ (n_1,n_2)\mapsto \left( \left\lceil \frac{n_1}{2} \right \rceil+1, \left\lfloor \frac
    {n_1}{2} \right \rfloor+1,n_2\right)\]
    is injective. 

Therefore, we have the following assignments and theorem.

\begin{theorem}
    For $ GL(3,\HH) $, the following assignments are the bijection in $
    (1): $
    \begin{align*}
        (\OO_0,(a,b,c))&\mapsto (a+4,b+2,c) \\
        (\OO_{mid}, (n_1,n_2))& \mapsto {%
        \begin{cases}
            \left(n_2+2,\left \lceil \frac{n_1}{2} \right\rceil, \left
            \lceil \frac{n_1-1}{2} \right\rceil\right) & n_1\leq 2n_2+1
            \\
            \; \\
            \left(\left \lceil \frac{n_1}{2} \right\rceil+1, \left
            \lceil \frac{n_1-1}{2} \right\rceil+1, n_2\right) & n_1>2n_2+1
        \end{cases}
        } \\
        (\OO_{prin},n)&\mapsto \left( \left \lceil \frac{n}{3} \right\rceil,
        \left \lceil \frac{n-1}{3} \right\rceil, \left \lceil \frac{n-2}
        {3} \right\rceil \right).
    \end{align*}
\end{theorem}

\begin{proof}
    The final piece to check is that this map is indeed a bijection.  We
    can classify all of the $ K $-dominant $ T $-weights in the
    following way.  Let $ (x,y,z) $ the weight.  The options for the
    differences $ x-y $ and $ y-z $ are first split into two cases:
    both $ x-y $ and $ y-z $ are $ \geq 2 $ or not.  In the affirmative,
    the only option is to be a weight attached to the zero orbit.  The
    map on the zero orbit is injective as the inputs there must satisfy $
    a\geq b\geq c\geq 0. $

    Suppose then that the $ x-y\geq 2. $ Necessarily the other
    difference must be $ \leq 1 $ and $ x\geq 2 $ by the dominance
    condition.  This immediately rules out the pairs coming from the
    principal orbit as the differences in weights there are identically $
    0 $ or $ 1. $ Thus, this weight must live on the middle orbit.  In
    particular, it must be mapped to by a pair $ (n_1,n_2) $ with $ n_1\leq
    2n_2+1. $ However, in this regime, the map is injective clearly.  A
    similar argument reduces the case of $ x-y\leq 1 $ and $ y-z\geq 2 $
    to the other regime for the middle orbit.  This is also injective
    and thus the entire map is.

    Surjectivity follows at once from the consideration of these cases.

\end{proof}

\subsection{The Subregular Orbit Generally}
The features of the middle orbit for $GL(3,\HH)$ turn out to (almost) give the bijection 
in general on the subregular orbit for $n$ odd. The $n$ even case turns out to
 need a slight modification from the principal orbit case. 
So, consider the orbit with given partition $[n-1,1].$ 

\subsubsection{The Odd Case} 
Recall from \cite{Collingwood1993}, that for $Z=SL(n,\C)$ and a
 partition $\lambda=[k_1^{a_1},...,k_\ell^{a_\ell}]$ of $n,$ there is a $Z$-orbit consisting of nilpotent elements
  with Jordan decomposition corresponding to $\lambda.$ By the Jacobson-Morozov theorem, we may 
  complete a choice of $X$ in this orbit to an $\lie{sl}(2,\C)$-triple with the semisimple
element to be composed of diagonal blocks of the form 
\[ \Mat{k-1 & 0 & 0 &... & 0 & 0\\
	0 & k-3 & 0 & ... & 0 & 0 \\
	0 & 0 & k-5 & ... & 0 & 0 \\
	\vdots  & \vdots & \vdots & \vdots & \vdots & \vdots \\
	0 & 0 &0 & ... & -k+3 & 0 \\
	0 & 0 & 0 & ... & 0 & -k+1 }.\] 
for each part of size $k.$ 

For the partition $[n-1,1]$ of $n$ we consider the associated partition $[(n-1)^2,1^2]$ of $2n$
 which corresponds to a nilpotent $GL(2n,\C)$-orbit\footnote{These are the same as $SL(n,\C)$-orbits.}. 
 As we assume $n$ is odd, $n-1$ is even and thus $n-2$ (the first diagonal entry in the semisimple 
 element $H$) is odd. The element $H$ thus is diagonal (up to reordering) of 
 the form \[ H=\operatorname{diag}(n-2,n-2,n-4, n-4,...,1,1,0,0,-1,-1,...,-(n-2),-(n-2)).      \]
 Therefore, there is a weight $1$ $\ad H$-eigenspace and this orbit is odd (in the sense of \cite{Collingwood1993}). 
 
The weight $1$ space is spanned by root subspaces corresponding to the
 roots \[e_{n-2}-e_n, e_{n-1}-e_n,e_{n-2}+e_n,e_{n-1}+e_n.\]  Using our
  general formulation of the pushforward and the results from the middle orbit 
  for $GL(3,\HH),$ we see that two subsets are of particular interest: $I_1=\{e_{n-2}-e_{n},e_{n-1}-e_n\}$ and 
  $I_2=\{e_{n-1}-e_n,e_{n-1}+e_n\}.$ The general form of the pushforward will have 16 terms in this situation 
  and the tempiric representations corresponding to the subsets $I_1$ and $I_2$ will have the maximal 
  infinitesimal character depending on the input weights. 
  
  Notice that for a choice of $X$ compatible with this $H,$ we obtain
   $K^X_{red}\simeq SL(2,\C)^2,$ $L\simeq GL(2,\C)^n,$ $L\cap K\simeq GL(2,\C)^{(n-1)/2}\times SL(2,\C),$
    and the restriction map is given as the identity on the second $SL(2,\C)$-factor.
     That is to say, $K^X_{red}$ is embedded in $L\cap K$ via a composition of the diagonal morphism
      of the first $SL(2,\C)$ into the $GL(2,\C)$ factors and the identity on the final $SL(2,\C)$-factor. 
      This induces a map on virtual modules given on weights 
      by \[ (a_1,b_1,a_2,b_2,...,a_{(n-1)/2)},b_{(n-1)/2},c)\mapsto (\sum_{i} (a_i-b_i),c).\]
      
Let $(a;b)$ be the highest weight of an irreducible $K^X$-module $\tau$. 
Consider the lifts of $\tau$: \[ \tilde{\tau}=\left( \left\lceil \frac{a}{n-1} \right \rceil+k_1, -\left\lceil \frac{a-1}
    {n-1} \right \rceil+k_1 ,...,-\left\lceil \frac{a-(n-2)}{n-1} \right \rceil +k_{(n-1)/2}, b \right).\]
    
Writing this as a sum of virtual $(\lie{l},L\cap K)$-modules we see that the
term with the largest (in norm) infinitesimal character appearing is given as 
\[ X\left( \left\lceil \frac{a}{n-1} \right \rceil+k_1+1,- \left\lceil \frac{a-1}
    {n-1} \right \rceil+k_1-1 ,... ,\left\lceil \frac{a-(n-2)}{n-1} \right \rceil +k_{(n-1)/2}-1, b+1 \right)\]
    where the tuple appearing in the brackets is the infinitesimal character. 
    
To properly apply cohomological induction, we need to compute $\rho(u)|_T.$
 As the parabolic here is isomorphic to the parabolic attached to the principal orbit,
 we see that \[\rho(u)|_T=(2n-2,2n-2,2n-6,2n-6,...,4,4,0).\]
 Unlike the even orbit cases, we have the twist by the sums of $T$-weights on $\lie{s}_1$. 
 For this reason, set $k_i=-(2n-4i-2)$ for $i<(n-1)/2$ and $k_{(n-1)/2}=-3.$ Suppose that $a\leq (n-1)b+(n-2).$ 
 Then for the subset $I_1$, we obtain a tempered representation 
 with largest (in norm) infinitesimal character: 
  \begin{equation}\label{Subregular_odd_1}   \left(b+3,\left\lceil  \frac{a}{n-1} \right \rceil+1, -\left\lceil \frac{a-1}
    {n-1} \right \rceil-1 ,... ,-\left\lceil \frac{a-(n-2)}{n-1} \right \rceil-1 \right).  \end{equation}

When $a>(n-1)b+(n-2)$ we have that the largest tempered representation appearing
corresponds to the subset $I_2.$ This has infinitesimal character: 
  \begin{equation}\label{Subregular_odd_2}  \left(\left\lceil  \frac{a+2}{n-1} \right \rceil+1, -\left\lceil \frac{a+1}
    {n-1} \right \rceil-1 ,... ,-\left\lceil \frac{a-(n-4)}{n-1} \right \rceil-1,b+1 \right).  \end{equation} 
This finishes the $n$ odd case computation. 

\subsubsection{The Even Case}

When $n=2m$ is even, we are in a completely new situation. The partition $[n-1,1]$ has
 a corresponding weighted Dynkin diagram with weights $2$ for all nodes except for
  a single node with weight $0.$ For this reason, the orbit is even \cite{Collingwood1993}!
   Using a similar $H$ as in the previous case, we have
    \[ H=\operatorname{diag}(n-2,n-4,...,0,0,0,0,...,-(n-4),-(n-2)).\]
Picking $X$ compatible with this $H,$ we have 
that $K^X_{red}\simeq SL(2)\times Sp(4,\C), L\simeq GL(2,\C)^{n-2}\times GL(4,\C)$
 and $L\cap K\simeq GL(2,\C)^{m-1}\times Sp(4,\C).$ The $SL(2,\C)$-factor of $K^X$
  is embedded diagonally in the $GL(2,\C)$-factors of $L\cap K.$   
 
 Our lifting rule will be (un)surprisingly similar to the odd case. Consider
  the highest weight of an irreducible $L\cap K$-representation $(a_1,b_1,...,a_{m},b_m; x,y).$ 
  Then the image of this under the restriction map is
   \[  (a_1,b_1,...,a_{m},b_m; x,y)\mapsto \left(\sum_i (a_i-b_i); x,y)\right) \] we thus
    see that the restriction of irreducible $L\cap K$  representations is an irreducible $K^X_{red}$ 
    module.    
   
   To compute the lift, we proceed as above. Given triple of non-negative
    integers $(a;b,c)$ with $b\geq c$ realized as the highest weight of
    an irreducible $K^X_{red}$ representation $\tau$. Then we have  lifts 
    \[ \widetilde{\tau}=\left( \left\lceil \frac{a}{n-2} \right \rceil+k_1, -\left\lceil \frac{a-1}
    {n-2} \right \rceil+k_1 ,...,-\left\lceil \frac{a-(n-3)}{n-2} \right \rceil +k_{m}, b,c \right).\]
Applying the Weyl character formula and the Cousin resolution, 
we obtain a sum of tempiric $(\lie{l},L\cap K)$-modules with the largest 
infinitesimal character appearing being \[ \left( \left\lceil \frac{a}{n-2} \right \rceil+k_1+1, -\left\lceil \frac{a-1}
    {n-2} \right \rceil+k_1-1 ,...,-\left\lceil \frac{a-(n-3)}{n-2} \right \rceil +k_{m}-1, b+3,c+1 \right).\]

We now need to compute $\rho(\lie{u})|_T.$ We can utilize our computation
 from above, for $L\simeq GL(2,\C)^n.$ The difference between $\rho(u)|_T$ for this parabolic
 and the one attached to the subregular orbit is the inclusion of the half sum of the $\lie{h}$-weights 
 \[e_{n-1}-e_{n+1},e_{n-1}-e_{n+2}, e_n-e_{n+1},e_{n}-e_{n+2}\]
 This implies that \[ \rho(\lie{u})|_T=(2n-2,2n-2,2n-6,2n-6,...,6,6,0,0).\]
 
Now we can apply cohomological induction. As $\rho(\lie{u})|_T$ is dominant, the application 
of cohomological induction preserves the ``largest tempiric" property of the sum. Thus, we see
that the largest tempiric ($\lie{g},K)$-module appearing in the expansion of $M(\mathbb{O}_{sr},\tilde{\tau})$ is 
given by the character \[ \left( \left\lceil \frac{a}{n-2} \right \rceil+k_1+(2n-1),...,-\left\lceil \frac{a-(n-3)}{n-2} \right \rceil +k_{m}+5, b+2,c \right).\]
    Hence, for $k_i=-(2n-4i-2)$ we minimize this character. 
    
The final part of the computation is to determine the $W_K$-dominant 
conjugate of these characters. This comes in three parts: 
\begin{description} 
	\item[\textbf{I})] $a\leq (n-2)c.$
	\item[\textbf{II})] $(n-2)c<a\leq (n-2)(b+2).$
	\item[\textbf{III})] $a>(n-2)(b+2).$  
\end{description} 
For case $\textbf{I},$ the lowest $K$-type of the corresponding tempiric module
 has highest weight \begin{equation}\label{Subregular_Even_1} \left(b+2,c, \left\lceil \frac{a}{n-2} \right \rceil, \left\lceil \frac{a-1}
    {n-2} \right \rceil,...,\left\lceil \frac{a-(n-3)}{n-2} \right \rceil\right).\end{equation} 

For case $\textbf{II},$ the lowest $K$-type of the corresponding tempiric module
 has highest weight \begin{equation}\label{Subregular_Even_2} \left(b+2, \left\lceil \frac{a}{n-2} \right \rceil, \left\lceil \frac{a-1}
    {n-2} \right \rceil,...,\left\lceil \frac{a-(n-3)}{n-2} \right \rceil,c\right).\end{equation} 
    
 Lastly, case $\textbf{III}$ has the following corresponding highest weight 
 \begin{equation}\label{Subregular_Even_3} \left( \left\lceil \frac{a}{n-2} \right \rceil, \left\lceil \frac{a-1}
    {n-2} \right \rceil,...,\left\lceil \frac{a-(n-3)}{n-2} \right \rceil,b+2,c\right).\end{equation} 
    
We collect all of this in the following result: 

\begin{theorem} 
	The $LKT$ map on the subregular orbit for $GL(n,\HH)$ is 
	given by the formulas \ref{Subregular_odd_1},\ref{Subregular_odd_2} for $n$
	 odd and \ref{Subregular_Even_1},\ref{Subregular_Even_2},\ref{Subregular_Even_3} 
	 when $n$ is even. 
\end{theorem} 
   
This completes our study of the $LKT$ map for now. It is clear
 that a more general approach is desirable as these computations
  only become more involved as the stabilizers grow. It is believed
   that the methods of Bezrukavnikov may lead somewhere for
    $GL(n,\HH)$ (or any group with suitably nice geometry, i.e. QCT),
     but we are unsure how far it will go.

\bibliographystyle{alphaurl} 
\bibliography{Article_References_Master}

\end{document}